\documentclass[reqno]{amsart}
\usepackage{amsmath}
\usepackage{amssymb}
\usepackage{graphicx}
\usepackage{tikz}
\usetikzlibrary{calc}
\usepackage{xcolor}
\usepackage{amsthm}
\usepackage{esint}
\usepackage{subcaption}

\theoremstyle{plain}
\begingroup
\newtheorem{theorem}{Theorem}[section]
\newtheorem{lemma}[theorem]{Lemma}
\newtheorem{proposition}[theorem]{Proposition}

\endgroup

\theoremstyle{definition}
\begingroup
\newtheorem{definition}[theorem]{Definition}

\endgroup

\usepackage[english]{babel}

\usepackage{comment}

\usepackage[english]{babel}

\begin{document}

\author{Paolo Piovano}
\address[Paolo Piovano]{University of Vienna, Oskar-Morgenstern-Platz 1, 1090 Vienna, Austria.}
\email{paolo.piovano@univie.ac.at}

\author{Leonard Kreutz}
\address[Leonard Kreutz]{Applied Mathematics M\"unster, University of M\"unster, 
Einsteinstrasse 62, 48149 M\"unster, Germany.}
\email{leonard.kreutz@gmail.com}

\title[Microscopic validation of thin-film models]{Microscopic validation of a variational model of epitaxially strained crystalline films}

\maketitle

\begin{abstract}
A discrete-to-continuum analysis for free-boundary problems related to crystalline films deposited on substrates is performed by $\Gamma$-convergence.
The discrete model here introduced  is characterized by an energy with two contributions, the surface and the elastic-bulk  energy, and it is formally  justified starting from  atomistic interactions. The surface energy counts missing bonds at the film and substrate boundaries, while the elastic energy models the fact that for film atoms there is a preferred interatomic distance different from the preferred interatomic distance for substrate atoms. 
In the regime of small mismatches  between the film and the substrate optimal lattices,  a discrete rigidity estimate is established by regrouping  the  elastic energy in triangular-cell  energies 
and  by locally applying rigidity estimates from the literature. This is crucial to establish  pre-compactness for  sequences  with  equibounded energy and to prove that the limiting deformation is one single rigid motion.  By properly matching the convergence scaling of the different  terms in the discrete energy, both surface and elastic contributions appear also in the resulting continuum  limit in agreement (and in a  form consistent)  with literature models. Thus, the analysis performed  here is  a microscopical justifications of such models. 
\end{abstract}


\medskip


\section{Introduction}

 In this paper a discrete model for crystalline films deposited on substrates in the presence of a mismatch between the parameters of the film and the substrate crystalline lattices  is introduced  and a discrete-to-continuum passage is performed by $\Gamma$-convergence. Thin films find nowadays an ever-growing number of applications, such as for optoelectronics, photovoltaic devices, solid oxide fuel/hydrolysis cells, and any advancement in the modeling of thin films can, in principle, induce a significant technological innovation. As the obtained continuum $\Gamma$-limit of our analysis is in accordance with the theory of Stress-Driven Rearrangement Instabilities (SDRI) (see \cite{AT,Gr}), and in particular with the variational  thin-film  models introduced  in \cite{DP1,fonseca2007equilibrium, spencer1999asymptotic,spencer1997equilibrium}, 
our investigation represents also a microscopical justification of such models. 

\hspace{0.4cm} The Transition-Layer model and the Sharp-Interface model for epitaxially strained films introduced in  \cite{spencer1999asymptotic, spencer1997equilibrium} for regular film profiles, and then analytically derived in \cite{DP1,fonseca2007equilibrium} for general profiles by $\Gamma$-convergence and by relaxation,  are characterized by energies displaying both surface and elastic-bulk terms.  Including in the model elastic bulk deformations is crucial in the presence of a mismatch between the crystalline lattices of the film and the substrate.   In fact,  even  though the minimum energy configuration for the bulk occurs at a stress-free structure for each material, the relaxation to the  respective elastic  equilibria of the two materials leads  to a discontinuous crystalline structure that is associated to  an extremely high energy contribution  at the interface between the film and the substrate. 
Thus, in order  to match the crystalline lattices, bulk deformations and mass rearrangement take place  in the bulk. As a consequence, discontinuities, such as corrugations or cracks, can be induced at the film profile. This is possible to some extent, as they have an energetic price in terms of the surface tension that is (at least in the isotropic case)  proportional to the surface area. The resulting configuration is therefore a compromise between the roughening effect of the elastic energy and the regularizing effect of the surface energy, with the former prevailing when   the thickness of the film is large enough  as discussed in \cite{fusco2012equilibrium}. 

\hspace{0.4cm}  In the same spirit of the SDRI theory  \cite{AT,Gr}  and of \cite{spencer1999asymptotic, spencer1997equilibrium}, also the discrete variational model here introduced is characterized by an energy $E_{\varepsilon}$ that displays both a surface term $E_{\varepsilon}^{S}$ and a elastic term $E_{\varepsilon}^{el}$, i.e., 
\begin{equation}\label{energy_intro}
E_{\varepsilon}(y,h):=E_{\varepsilon}^{S}(y,h)+E_{\varepsilon}^{el}(y,h), 
\end{equation}
where $\varepsilon>0$ is a scaling parameter, $h$ measures the hight of the film, and $y$ represents the  bulk deformation. More precisely, in the discrete setting we fix a reference lattice $\mathcal{L}$, here chosen to be a equilateral triangular lattice, and $E_{\varepsilon}$ depends on discrete functions $h$ and $y$ denoted \emph{discrete profiles} and  \emph{discrete deformations}, respectively, that are defined with respect to $\varepsilon\mathcal{L}$. Given the discrete set  $S_\varepsilon:=\varepsilon\mathcal{L}\cap((0,L)\times\{x_2=0\})$  for  a parameter $L>0$ fixed,  a film discrete profile is a function $h : S_\varepsilon\to \mathbb{R}_+$ such that  the elements in $\{x=(i,x_2)\in\varepsilon\mathcal{L} : x_2\in(0,h(i))\}$ are assumed to be film atoms for every  $i\in S_\varepsilon$. We also identify each discrete profile $h$ with a properly characterized lower-semi\-con\-ti\-nuous piecewise constant interpolation over the interval $(0,L)$ so that, 
$$\Omega_h^+ :=  \left\{(x_1,x_2)\in\mathbb{R}^2 : 0<x_1<L, 0<x_2 <  h(x_1)  \right\}$$
 represents the region occupied by films with profile $h$. The substrate is instead assumed to occupy the region $\Omega^-:=(0,L)\times(-R,0]\subset\mathbb{R}^2$ for some parameter $R>0$, with $\varepsilon\mathcal{L}\cap\Omega^-$  representing   the  reference  substrate atoms.  
 Furthermore, we call  discrete deformation every function $y : \mathcal{L}_\varepsilon(\Omega_h) \to \mathbb{R}^2$ where $\mathcal{L}_\varepsilon(\Omega_h):= \mathcal{L}_\varepsilon \cap \Omega_h$ with $\Omega_h:=\Omega_h^+\cup\Omega^-$. 

\hspace{0.4cm} The   surface energy $E_{\varepsilon}^{S}$ in \eqref{energy_intro} takes into account the missing bonds at the boundary of $\Omega_h$ and it is defined for  each discrete profile $h$ by
\begin{equation}\label{energy_surface_intro}
E_{\varepsilon}^{S}(h):= 
\sum_{x \in  \mathcal{L}_\varepsilon(\Omega_h)}\hspace{-0.1cm}\varepsilon\gamma(x)(6- \#\mathcal{N}_\varepsilon(x)),
\end{equation}
where 
$\mathcal{N}_\varepsilon(x) = \left\{\tilde{x} \in \mathcal{L}_\varepsilon(\Omega_h)\setminus \{x\} : |\tilde{x}-x|\leq \varepsilon \text{ or }  |\tilde{x}\pm  (L,0) -x|\leq \varepsilon \right\}$ denotes the set of nearest neighbors  (with $L$-periodic condition) of $x\in\mathcal{L}_\varepsilon(\Omega_h)$,  and 
\begin{align*}
\gamma(x):=  \begin{cases} \gamma_f&\textrm{if $x \in \Omega_h^+$},\\
\gamma_s&\textrm{if $x \in  \Omega^-$}
\end{cases}
\end{align*} 
for  some positive constants $\gamma_f$ and $\gamma_s$ depending on the film and substrate material, respectively. 
Notice that it would be equivalent to introduce  a dependence on  discrete deformations $y$ also in the definition of $E_{\varepsilon}^{S}$ by considering the sum in \eqref{energy_surface_intro} as extended over the elements of the deformed lattice $y(\mathcal{L}_\varepsilon(\Omega_h))$, if we  restrict to \emph{small deformations} $y$, i.e., deformations $y$ that do not change the topology of the lattice or, in other words, such that  $\#\mathcal{N}_\varepsilon(y(x))=\#\mathcal{N}_\varepsilon(x)$ for every $x\in\mathcal{L}_\varepsilon(\Omega_h)$. 

\hspace{0.4cm} The $\varepsilon$-rescaled elastic energy $E_{\varepsilon}^{el}$  in \eqref{energy_intro}, which model the elastic energy contribution due to elongation and compression of bonds, is defined on pairs  $(y,h)$  by 
\begin{equation}\label{energy_intro}
E_{\varepsilon}^{el}(y,h):= 
\sum_{x \in \mathcal{L}_\varepsilon(\Omega_h)}\sum_{\tilde{x} \in \mathcal{N}_\varepsilon(x)}\hspace{-0.1cm} \varepsilon V_{x,\tilde{x}}^\varepsilon\left(\frac{|y(x)-y(\tilde{x})|}{\varepsilon} \right),
\end{equation}
where  the potential $V_{x,\tilde{x}}:[0,\infty)\to\mathbb{R}$ is chosen to be a nonlinear elastic potentials attaining its minimum at different lengths for film or substrate atoms $x$. If $x\in\Omega^-$ the bonding equilibrium length is $\varepsilon$, while, if  $x\in\Omega^+_h$, it is  $\varepsilon \lambda_\varepsilon>0$. Therefore, the possibility of a nonzero \emph{lattice mismatch}
$$\delta_\varepsilon:=\lambda_\varepsilon-1$$
is taken into consideration. 

The aim of the paper is to pass to the limit of $E_{\varepsilon}$ in the sense of $\Gamma$-convergence \cite{braides2002gamma} with respect to a  proper topology and under the hypotheses that  the discrete profiles $h$ satisfy a  \emph{$\varepsilon$-volume  constraint}, i.e.,   $\|h\|_{L^1}=V_\varepsilon$ for some $V_\varepsilon \in \varepsilon^2 \sqrt{3}/2 \mathbb{N}$ chosen so that  $V_\varepsilon \to V>0$. Notice  that the $\varepsilon$-scaling of $E_{\varepsilon}^{el}$ is here chosen to properly match the convergence scaling of   $E_{\varepsilon}^{S}$ so that  both surface and elastic contributions appear in the resulting continuum  limit $E$  in agreement with  the SDRI theory  and literature models \cite{AT,DP1,fonseca2007equilibrium, Gr,spencer1999asymptotic,spencer1997equilibrium}.   


\hspace{0.4cm}  Rigorous $\Gamma$-convergence results \cite{dal2012introduction}  on the derivation of linear-elastic theories from  nonlinear elastic discrete energies  have been obtained in \cite{braides2007derivation,schmidt2009derivation} without the presence of surface energies.  A key  instrument in the analysis consists in establishing a \emph{discrete rigidity result} that allows to  globally   estimate the closeness of the deformation gradients  to a single rotation by only knowing that locally  deformation gradients  are close to the family of rotations, and to pass from the deformations $y$ to the rescaled ($\varepsilon$-dependent) associated displacements $u$.   For this reason we denote in the following \emph{discrete triples} those triples $(y,u,h)$ where $h$ is a discrete profile, $y$ is a discrete deformation, and $u$ is a \emph{discrete displacement} associated to $y$, i.e.,  $u:\mathcal{L}_\varepsilon(\Omega_h)\to\mathbb{R}^2$ is given by 
$u(x) = \varepsilon^{-1/2}(y(x)-(Rx+b))$ 
for some rotation $R \in SO(2)$ and $b \in \mathbb{R}^2$. The discrete rigidity estimate is obtained by regrouping  the  elastic energy in triangular-cell  energies  accounting for all the elastic contributions  related to each triangle in the film and substrate lattices, and by locally applying the rigidity estimate in  \cite{friesecke2002theorem} (see also \cite{schmidt2009derivation}).  For this to be implemented  the \emph{orientation-preserving condition}, 
$$\det\left(y(\tilde{x})-y(x),y(\bar{x})-y(x)\right)\det\left(\tilde{x}-x,\bar{x}-x\right) \geq 0,$$
as shown in \cite{alicandro2017effect},  must be imposed on the discrete displacements $y$ for every mutual nearest neighbors  $x,\tilde{x},\bar{x} \in \mathcal{L}_\varepsilon(\Omega_h)$, in order to avoid local reflection of the reference lattice, which might be otherwise energetically convenient (see \cite{braides2007derivation}).  Without such condition rigidity would fail and  a more complex description of the $\Gamma$-limit would be needed that do not find currently any examples  in the literature to the best of our knowledge. 

\hspace{0.4cm} Moreover, the discrete rigidity estimates  is  obtained  under the hypothesis that 
 \begin{equation}\label{mismatch_small_regime}
\delta_\varepsilon\varepsilon^{-1/2}\to \delta\in\mathbb{R}
 \end{equation}
 and hence,  from the three \emph{lattice-mismatch regimes} that we can characterize: 
\begin{itemize}
\item[a)] $|\delta_\varepsilon| < \varepsilon^{1/2}$,
\item[b)] $|\delta_\varepsilon| \sim \varepsilon^{1/2}$,
\item[c)] $|\delta_\varepsilon| > \varepsilon^{1/2}$, 
\end{itemize} 
we are here treating only a) and b). This  is due   to the fact that the rigidity estimate allows  to linearize around one single rotations both for the substrate and the film (by also introducing a tensor, the so-called  \emph{mismatch strain},  that accounts for the error due to the mismatch in equilibrium length). In regime c) where lattice mismatches are very large, imposing smooth deformations  on a reference lattice without  dislocations, i.e., extra half-lines of atoms appearing in one of the lattice directions,   creates a too high  energy contribution at the film-substrate interface (in fact infinite) \cite{fanzon2017variational,fonseca2018model}. Therefore,  the basic  modeling  assumption here considered of describing both film and substrate lattices as parametrized through deformations  on the same periodic reference lattice  seems  not feasible in regime c). 
We refer the reader to  \cite{FFLM4,HMRG}  for experimental evidence, especially for films with larger thickness,  of formation of dislocation-patterns  as a further mode of strain relief, and to \cite{lazzaroni2013discrete} where the presence of dislocations is taken into account  in a discrete-to-continuum passage for a model of nanowires (in the absence of a free boundary for one of the crystalline phases).

\hspace{0.4cm} In order to perform a discrete-to-continuum analysis it is convenient to embed \emph{discrete triples} $(y,u,h)$  to the larger configurational space
 $$
X:=\left\{(y,u,h) \,: \text{  $h$ is l.s.c. with  $\mathrm{Var}h<+\infty$, and  $y,u \in L^2_{\mathrm{loc}}(\Omega_h;\mathbb{R}^2)$} \right\}
 $$
by identifying not only each $h$ with proper piecewise interpolations in $(0,L)$, but also  $y$ and $u$ with properly defined piecewise affine interpolations in $\Omega_h$, and by extending $E_\varepsilon$ to be $+\infty$  on non-discrete  triples in $X$.  
In particular, we consider in $X$ the metrizable topology $\tau_X$ associated to the  notion  of convergence:  $(h_\varepsilon,y_\varepsilon,u_\varepsilon)\to(y,u,h) \in X$   if and only if   $y_\varepsilon \to y$ and $u_\varepsilon \to u$ in $L^2_{\mathrm{loc}}(\Omega_h;\mathbb{R}^2)$, and $\mathbb{R}^2 \setminus \Omega_{h_\varepsilon}$ converge to $ \mathbb{R}^2 \setminus \Omega_h$ with respect to the  Hausdorff-distance  (see Definition \ref{Definition Convergence}).   
 The pre-compactness with respect to this topology in X of energy equi-bounded sequences, i.e., sequences $(h_\varepsilon,y_\varepsilon,u_\varepsilon)$ such that $\sup\{E^\varepsilon(h_\varepsilon,y_\varepsilon,u_\varepsilon) : \varepsilon>0\}<\infty$, follows (apart from redefining the displacements associated to $y_\varepsilon$) from  the rigidity estimate. 
 However,  the fact that admissible  profile functions  $h$ are  in general  not Lipschitz represents a further difficulty as it is not guaranteed that the rigidity estimate of  \cite{friesecke2002theorem} can be applied uniformly. 
It is only possible to apply it for sets $\tilde{\Omega}$ that are smooth and compactly contained in $\Omega_{h_\varepsilon}$ for $\varepsilon$ small enough.   In  view of the discrete graph constraint  though, we are then able   to  invade the film region  and obtain a rigidity result with a fixed rotation on the substrate as well as the film. 

\medskip

\hspace{0.4cm}  We now  discuss the interesting features of the limiting energy  $E$  that  is defined by
\begin{align*}
E(y,u,h) := \begin{cases} E^{el}(y,u,h)+E^{S}(y,u,h)&\,\, \text{if $\|h\|_{L^1} = V$ and $y = Rx +b$}\\
 &\,\, \text{for some $(R,b)\in SO(2)\times \mathbb{R}^2$}\\
+\infty&\, \text{otherwise,}
\end{cases}
\end{align*}
for every $(y,u,h)\in X$, where $E^{el}$ and $E^{S}$ denote the elastic and surface energy, respectively.  The elastic  energy $E^{el}$ is  given by  
\begin{align}\label{elastic_continuum_energy}
 E^{el}(y,u,h) =\int_{\Omega_h} W_y(x,Eu(x))\mathrm{d}x
 \end{align}
where the elastic density is $W_y(x_2,A):= \frac{16}{\sqrt{3}}K(x_2) \left(|A-E_0(x_2,y)|^2+\frac12 Tr^2(A-E_0(x_2,y))\right)$ for
$$
K(x_2):=
\begin{cases}
K_f&\text{if }x_2> 0,\\
K_s&\text{if }x_2\leq0,
 \end{cases}
$$
and for the mismatch strain  
$$
E_0(x_2,y):=
\begin{cases}
\delta\nabla y&\text{if }x_2> 0,\\
0&\text{if }x_2\leq0.
 \end{cases}
$$ 
We observe that the elastic energy density  $W_y(x_2,\cdot)$ corresponds to linear elastic isotropic materials with Lam\'e parameters $\lambda_\alpha$ and $\mu_\alpha$ for $\alpha=f,s$ referring, respectively, to the film and the substrate, with $\lambda_\alpha=\mu_\alpha$. We notice that Lam\'e coefficients independent from each other could be obtained for the $\Gamma$-limit by  considering in the discrete model also  longer range interactions or changing the reference lattice.   Furthermore, we observe that  the elastic energy density  depends  on both $x_2$ and $y$.  The $y$-dependence in $W_y(x_2,\cdot)$ is due to the fact that the mismatch strain needs   to be measured with respect to the limiting orientation of the reference lattice (our discrete energies are frame indifferent),  while the   $x_2$-dependence is related to the fact that the mismatch strain is only nonzero in the film region, where the atoms in the reference lattice are not at the optimal film bonding distance.   Moreover, it is for the  $y$-dependence that it is necessary to double the  elastic  variables already  at  the discrete level, and keep track of the deformations $y$ as well as the associated rescaled displacements $u$. In the limit their dependence decouples and they are independent from each other, but note that their energy is not.

\hspace{0.4cm} The surface energy $E^{S}$ is defined by
  \begin{align}
 E^S(h) &= \gamma_f\left(\int_{\partial \Omega_h \cap \Omega_h^{1/2}\cap \{x_2>0\}} \varphi(\nu)\mathrm{d}\mathcal{H}^1+2\int_{\partial \Omega_h \cap \Omega_h^{1}} \varphi(\nu)\mathrm{d}\mathcal{H}^1\right)\nonumber\\&\qquad\qquad\qquad\qquad\qquad\qquad+\gamma_s\wedge\gamma_f\int_{\partial \Omega_h \cap \Omega_h^{1/2}\cap \{x_2=0\}} \varphi(\nu)\mathrm{d}\mathcal{H}^1 \label{surface_continuum_energy}
 \end{align}
for some constants $\gamma_f, \gamma_s>0$, the he anisotropic surface tension 
 $$\varphi(\nu) := 2\sqrt{3}/3\left(\left| \nu_2\right| + 1/2\left|\sqrt{3}\nu_1 -\nu_2\right| + 1/2\left|\sqrt{3}\nu_1 + \nu_2 \right| \right),$$
 
and the sets  $\Omega^s_h$ denoting the set of points of $\Omega_h$ with  density $s \in [0,1]$. We observe that  $\varphi$ depends on the choice of the reference lattice. Furthermore, the first term in the surface energy relates to the essential boundary of $\Omega_h$ and appears  with the factor $\gamma_f$, while the second, which relates to the cuts in the boundary in $\Omega_h$,  presents the factor $2\gamma_f$ since cuts need to be  counted double as they correspond   at  the discrete level to  cracks  of infinitesimal width.  Finally, the factor of the third term that is related to the portion of the boundary intersecting the substrate surface, distinguishes two regimes: in the \emph{wetting regime}, for $\gamma_s > \gamma_f$,   it is energetically more convenient to cover, namely to wet, the substrate with an infinitesimal layer of film atoms, while in the \emph{dewetting regime}, for $\gamma_s < \gamma_f$, it is better to leave the substrate exposed. 

 
 \medskip
 
%

In order to establish the $\Gamma$-convergence result we establish the lower bounds for the surface energy and the elastic energy separately.  For the surface energy we use the result in \cite{KP} in order to obtain a semi-continuity result for surface energies defined on sets. Furthermore the lower bound for the elastic energy is performed in the interior of the set $\Omega_h$, where the compactness and rigidity results ensure good convergence and equi-integrability properties. We then Taylor-expand our interaction energies close to the limiting deformation $y$. From this estimate in the interior we pass to a global estimate by invading the set $\Omega_h$.
 The upper bound for the limiting energy is obtained by performing careful density arguments.  The first one ensures that for every profile $h$ there exists a sequence Lipschitz profiles $\{h_k\}_k$ converging from below to $h$ in such a way that the  surface  energy converges. Here we use the Yosida transform of $h$ to obtain Lipschitz profiles $\{h_k\}_k$ (not satisfying the volume constraints) such that $E^S(h_k) \to E^S(h)$. The calculations to check this are much in the spirit of \cite{chambolle2002computing}. 
In this way also the elastic energy of the approximation (which is just the function restricted to $\Omega_{h_k}$) converges to the limiting elastic energy.  Once this sequence is constructed we still need to modify $h_k$ as well as $u$ (now depending on $k$) so that the sets $\Omega_{h_k}$ satisfy the volume constraints. This involves refining arguments used in \cite{fonseca2007equilibrium} or \cite{DP1} in order to deal with anisotropic surface energies.  Finally, we observe that for a Lipschitz profile $h$ and a general displacement $u$ with finite energy there exists a sequence $\{u_k\}_k \in C^\infty(\Omega_h)$ converging to $u$ in $L^2(\Omega_h)$ such that also the energy converges,  and for such $(u,h)$ we construct the recovery sequence explicitly. The recovery sequence is obtained  by interpolating $u$ at the lattice nodes, considering the piecewise constant interpolations of $h$, and then rising them to match the volume constraint. 
 
 \medskip
 
 The article is organized as follows. In Section 2 we introduce the  mathematical setting and the discrete model.  In Section 3  we prove coercivity of our energies.  Finally, in Section 4  we perform the asymptotic analysis by proving the lower and the upper bound.

\section{ The mathematical setting}\label{setting}

 In this section we introduce the main notation and the mathematical setting of the discrete and continuum models. 
 We begin by  recalling   the definition of $\Gamma$-convergence that represents the main instrument  to derive effective theories for discrete systems (see \cite{braides2002gamma,dal2012introduction} for a detailed introduction to the theory of $\Gamma$-convergence). 

\medskip
 Given a metric space $(X,d_X)$ and a sequence of functionals $\{E_\varepsilon\}: X \to [0,+\infty]$, we say that  $E_\varepsilon$ \emph{$\Gamma$-converge} to a functional $E: X \to [0,+\infty]$  
if the following two conditions hold:
\begin{itemize}
\item[(i)] For all $\{x_\varepsilon\}_\varepsilon \subset X$ converging to $x\in X$ with respect to $d$ there holds
\begin{align*}
\liminf_{\varepsilon \to 0} E_\varepsilon(x_\varepsilon) \geq E(x).
\end{align*}
\item[(ii)] For all $x \in X$ there exists a sequence  $\{x_\varepsilon\}_\varepsilon $ converging to $x$ with respect to $d$ such that
\begin{align*}
\limsup_{\varepsilon \to 0} E_\varepsilon(x_\varepsilon)\leq E(x).
\end{align*}
\end{itemize}
In this case we write $\displaystyle E= \Gamma\text{-}\lim_{\varepsilon \to 0} E_\varepsilon$.
 Furthermore,  we consider the functionals $E^\prime,E^{\prime\prime}: X \to [0,+\infty]$  defined by 
\begin{align*}
E^\prime(x) := \inf \left\{\liminf_{\varepsilon \to 0}E_\varepsilon(x_\varepsilon) : d_X(x_\varepsilon,x) \to 0 \right\}
\end{align*}
and 
\begin{align*}
E^{\prime\prime}(x) := \inf \left\{\limsup_{\varepsilon \to 0}E_\varepsilon(x_\varepsilon) : d_X(x_\varepsilon,x) \to 0 \right\}
\end{align*}
for every $x \in X$, respectively.   We observe that $E^\prime, E^{\prime\prime}$ always exist, and  that there exists $E: X \to [0,+\infty]$ with $E^\prime=E^{\prime\prime} = E(x)$ for all $x \in X$ if and only if $\displaystyle E= \Gamma\text{-}\lim_{\varepsilon \to 0} E_\varepsilon$ \cite{braides2002gamma,dal2012introduction}.  In the following we denote $E^\prime$ and $E^{\prime\prime}$ by  $ \Gamma\text{-}\liminf_{\varepsilon \to 0} E_\varepsilon$ and $ \Gamma\text{-}\limsup_{\varepsilon \to 0} E_\varepsilon$, respectively.  

\medskip   The main result of the paper is a $\Gamma$-convergence result (see Theorem \ref {main_theorem}). 
 In the rest of this section we introduce the metric space $(X,d_X)$, the sequence $\{E_\varepsilon\}$, and the limit functional $E$ for which we establish the $\Gamma$-convergence results.  Notice that this result represents a \emph{discrete-to-continuum passage} as the functionals $E_\varepsilon$ are defined with respect to reference lattices $\mathcal{L}_\varepsilon$, while the $E$ depends on functions defined on  continuum sets. 

 To this aim let us recall from the Introduction some notation. We fix two parameters $L,R>0$ and we denote the scaling parameter by $\varepsilon>0$. In the following, we denote by $C>0$ a generic constant that may change from line to line. We consider $S^1_L = \mathbb{R}/L\mathbb{Z}$ endowed with its usual distance, $Q = (0,L) \times (-R,+\infty)$,  and $Q^+ = (0,L) \times (0,+\infty)$.  For $A,B \subset \mathbb{R}^2$ we denote the Hausdorff-distance between the set $A$ and the set $B$ by
$$
\mathrm{dist}_{\mathcal{H}}(A,B) = \sup\left\{\sup\{\mathrm{dist}(x,B) :x \in A\},\sup\{\mathrm{dist}(x,A) : x\in B\}\right\},
$$ 
and by $SO(2) =\{R \in \mathbb{R}^{2\times 2} :  RR^T=\mathrm{Id},\mathrm{det}(R)=1\}$ the set of all rotations in $\mathbb{R}^2$.



\subsection{The discrete model}


In this subsection we define the discrete triples of profiles, deformations, and displacements, and the general configurational space with the topology with respect to the $\Gamma$-convergence is carried out, and we introduce the discrete energy (also starting form atomistic interactions).  
\medskip

 \textbf{Discrete profiles.}  
We say that  $h : \frac{\sqrt{3}}{2}\varepsilon(\mathbb{Z}+\frac{1}{2}) \cap S^1_L \to \mathbb{R}_+$ is a \emph{discrete profile} if 
\begin{align}\label{eq : depo}
h\left(\frac{\sqrt{3}}{2}\varepsilon \left(i +\frac{1}{2}\right)\right) \in \begin{cases} \frac{\varepsilon}{2} + \varepsilon\mathbb{N} &i \text{ even},\\
\varepsilon\mathbb{N} &i \text{ odd}.
\end{cases}
\end{align}
 Furthermore, we identify every discrete profile $h$ with the lower-semicontinuous piecewise-constant interpolation defined as  
\begin{align}\label{h Interpolation}
h(x) = \begin{cases}
 h(i)  & i-\frac{\sqrt{3}}{4}\varepsilon < x < i+\frac{\sqrt{3}}{4}\varepsilon,i \in \frac{\sqrt{3}}{2}\varepsilon(\mathbb{Z}+\frac{1}{2}) \cap S^1_L, \\
 \min\{h(i),h(i+\frac{\sqrt{3}}{2}\varepsilon)\} &x=i+\frac{\sqrt{3}}{4}\varepsilon, i \in \frac{\sqrt{3}}{2}\varepsilon(\mathbb{Z}+\frac{1}{2}) \cap S^1_L.
 \end{cases}
\end{align}
The  \emph{set  of admissible profiles}  is denoted by $AP(S^1_L)$ and characterized as 
\begin{align}\label{Space: AP}
AP(S^1_L) = \Big\{h: S^1_L \to \mathbb{R}_+ : h \text{ is lower-semicontinuous and } \mathrm{Var}h <+\infty\Big\}.
\end{align}
 For every profile  $h \in AP(S^1_L)$ we set
\begin{align}\label{Defintion Omega}
\begin{split}
&\Omega_h := \left\{(x_1,x_2) \in Q : x_2 <  h(x_1)  \right\},\\& \Omega^-= \left\{(x_1,x_2) \in Q: -R < x_2 \leq 0 \right\},\\&
\Omega_h^+= \Omega_h \setminus \Omega^-.
\end{split}
\end{align} 

 We notice that for every discrete profile $h$ we have  $h(0)=h(L)= \min\{h(\frac{\sqrt{3}}{4}\varepsilon),h( L -\frac{\sqrt{3}}{4}\varepsilon)\}$ and, by considering the identified interpolation \eqref{h Interpolation}, $h\in AP(S^1_L)$ (see Figure ~\ref{Fig.Reference Configuration}). 

\medskip

 \textbf{Reference lattice.}  
We choose a triangular lattice as the reference lattice. More precisely, let 
$$
A= \frac{1}{2}\begin{pmatrix}
\sqrt{3} & 0 \\
1 & 2
\end{pmatrix}
$$
and set
$\mathcal{L}_\varepsilon = \varepsilon\left(
A \mathbb{Z}^2+\left(\frac{\sqrt{3}}{4},0\right) \right)$. Furthermore, for any set $B\subset \mathbb{R}^2$ we denote $\mathcal{L}_\varepsilon(B) = \mathcal{L}_\varepsilon \cap B$. 

For $i \in \mathcal{L}_\varepsilon(\Omega_h)$ we define
\begin{align}\label{Definition Neighbourhood}
\mathcal{N}_\varepsilon(i) = \left\{j \in \mathcal{L}_\varepsilon(\Omega_h)\setminus \{i\} : |j-i|\leq \varepsilon \text{ or }  |j\pm  L   e_1 -i|\leq \varepsilon \right\}
\end{align}
the set of \textit{nearest neighbours of the point} $i \in \mathcal{L}_\varepsilon(\Omega_h)$. Note that by this definition points on the vertical boundary may be identified as neighbors,  so that also interactions across the lateral boundaries will be allowed. 
 Henceforth we assume that 
\begin{equation}\label{Llength}
 L =\sqrt{3} k_\varepsilon \varepsilon 
 \end{equation}
  for some $k_\varepsilon \in \mathbb{N}$ for all $\varepsilon >0$. Next we define the set of triangles of  $\mathcal{L}_\varepsilon(Q)$   by 
\begin{align*}
\mathcal{T}_\varepsilon = \Big\{&\{i_1,i_2,i_3\} : (i_1,i_2,i_3) \in (\mathcal{L}_\varepsilon(Q))^3,i_j \in \mathcal{N}_\varepsilon(i_k), \, j \neq k, \Big\}.
\end{align*}

\begin{figure}
\centering
\begin{tikzpicture}[scale=0.4]
\draw(-1,-1.5)--++(0,21);
\draw(17.5,-1.5)--++(0,21);
\draw[dashed](-.5,9.5)--++(17.5,0);
\draw(-1,9.5) node[anchor=east]{$y=0$};
\begin{scope}
\clip(-0.5,-0.25)--++(17.5,0)--++(0,10)--++(-17.5,0)--++(0,-10);
\foreach \j  in {-20,...,20}{
\foreach \i in {0,...,30}{
\draw[fill=gray](0,\j)++(30:\i) circle (.125); 
}
}
\end{scope}
\foreach \j in {0,...,8}{
\draw[fill=white](0,10+\j) circle (.125);
}
\foreach \j in {0,...,4}{
\draw[fill=white](0,10+\j)++(30:1) circle (.125);
}
\foreach \j in {0,...,6}{
\draw[fill=white](0,10+\j)++(30:1)++(-30:1) circle (.125);
}
\foreach \j in {0,...,2}{
\draw[fill=white](0,10+\j)++(30:2)++(-30:1) circle (.125);
}
\foreach \j in {0,...,2}{
\draw[fill=white](0,10+\j)++(30:2)++(-30:2) circle (.125);
}
\foreach \j in {0,...,6}{
\draw[fill=white](0,10+\j)++(30:3)++(-30:3) circle (.125);
}
\foreach \j in {0,...,6}{
\draw[fill=white](0,10+\j)++(30:4)++(-30:3) circle (.125);
}
\foreach \j in {0,...,5}{
\draw[fill=white](0,10+\j)++(30:4)++(-30:4) circle (.125);
}
\foreach \j in {0,...,4}{
\draw[fill=white](0,10+\j)++(30:5)++(-30:4) circle (.125);
}
\foreach \j in {0,...,3}{
\draw[fill=white](0,10+\j)++(30:5)++(-30:5) circle (.125);
}

\foreach \j in {0,...,6}{
\draw[fill=white](0,10+\j)++(30:6)++(-30:5) circle (.125);
}
\foreach \j in {0,...,6}{
\draw[fill=white](0,10+\j)++(30:6)++(-30:6) circle (.125);
}
\foreach \j in {0,...,7}{
\draw[fill=white](0,10+\j)++(30:7)++(-30:6) circle (.125);
}
\foreach \j in {0,...,5}{
\draw[fill=white](0,10+\j)++(30:7)++(-30:7) circle (.125);
}
\foreach \j in {0,...,4}{
\draw[fill=white](0,10+\j)++(30:8)++(-30:7) circle (.125);
}
\foreach \j in {0,...,4}{
\draw[fill=white](0,10+\j)++(30:8)++(-30:8) circle (.125);
}

\foreach \j in {0,...,2}{
\draw[fill=white](0,10+\j)++(30:9)++(-30:9) circle (.125);
}
\foreach \j in {0,...,7}{
\draw[fill=white](0,10+\j)++(30:10)++(-30:9) circle (.125);
}
\draw(-0.5*0.86602540378,18.5)--++(0.86602540378,0)--++(0,-3.5)--++(0.86602540378,0)--++(0,1.5)--++(0.86602540378,0)--++(0,-3.5)--++(0.86602540378,0)--++(0,-.5)--++(0.86602540378,0)--++(0,-2.5)--++(0.86602540378,0)--++(0,6.5)--++(0.86602540378,0)--++(0,.5)--++(0.86602540378,0)--++(0,-1.5)--++(0.86602540378,0)--++(0,-.5)--++(0.86602540378,0)--++(0,-1.5)--++(0.86602540378,0)--++(0,3.5)--++(0.86602540378,0)--++(0,-.5)--++(0.86602540378,0)--++(0,1.5)--++(0.86602540378,0)--++(0,-2.5)--++(0.86602540378,0)--++(0,-.5)--++(0.86602540378,0)--++(0,-.5)--++(0.86602540378,0)--++(0,-4.5)--++(0.86602540378,0)--++(0,2.5)--++(0.86602540378,0)--++(0,5.5)--++(0.86602540378,0);
\draw[<->](-.5,-1)--++(17.5,0);
\draw(8.5,-1) node[anchor=north]{$(0,L)$};
\end{tikzpicture}
\caption{The set $\Omega_h$ for a discrete profile $h$.  The black atoms are substrate atoms, whereas the white atoms belong to the film.}
\label{Fig.Reference Configuration}
\end{figure}
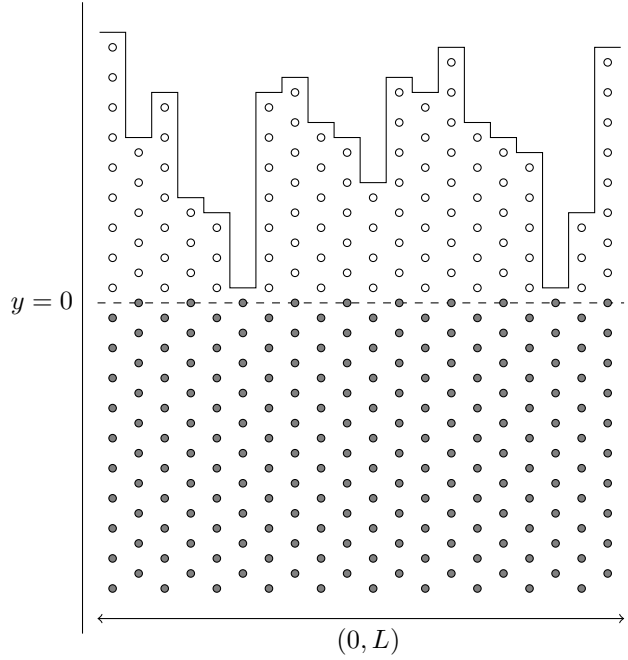

\
 \textbf{Discrete deformations.}  We refer to any map  $y : \mathcal{L}_\varepsilon(\Omega_h) \to \mathbb{R}^2$ as \emph{discrete deformation}, and we reduce to only  discrete deformations that are \textit{ orientation preserving}, i.e.,  such that 
\begin{align}\label{orientation preserving}
\det\left(y(i_2)-y(i_1),y(i_3)-y(i_1)\right)\det\left(i_2-i_1,i_3-i_1\right) \geq 0
\end{align}
for any $|i_k - i_j| = \varepsilon$ and $k\neq j$. 

\medskip
 \textbf{The discrete energy.}  For any discrete profile $h : \frac{\sqrt{3}}{2}\varepsilon(\mathbb{Z}+\frac{1}{2}) \cap S^1_L \to \varepsilon\mathbb{N}$ and orientation-preserving discrete deformation $y :\mathcal{L}_\varepsilon(\Omega_h) \to \mathbb{R}^2$ we define the energy of the pair $(y,h)$ by the sum 
\begin{align}\label{Definition Energy}
E_\varepsilon(y,h) = E^S_\varepsilon(h) + E^{el}_\varepsilon(y,h),
\end{align}
where 
\begin{align*}
&E^S_\varepsilon(h)= \gamma_f\sum_{i \in \mathcal{L}_\varepsilon(\Omega_h^+ )}\varepsilon(6- \#\mathcal{N}_\varepsilon(i)) + \gamma_s \sum_{i \in \mathcal{L}_\varepsilon(\Omega^-)}\varepsilon(6- \#\mathcal{N}_\varepsilon(i)), 
\end{align*} 
where $\gamma_f,\gamma_s >0$. 
The elastic energy of the system is given by
\begin{align}\label{Elasticrewritten}
E^{el}_\varepsilon(y,h) = \sum_{ i \in \mathcal{L}_\varepsilon(\Omega_h)} \sum_{j \in \mathcal{N}_\varepsilon(i)} \varepsilon V_{i,j}^\varepsilon\left(\frac{|y(i)-y(j)|}{\varepsilon} \right),
\end{align}
where   $V_{i,j}^\varepsilon  : \mathbb{R} \to \mathbb{R}$  is defined by
\begin{align}\label{DefinitionVi}
V_{i,j}^\varepsilon(r) = \begin{cases} \frac{K_s}{2}(r-1)^2 &i \in \mathcal{L}_\varepsilon(\Omega^-),|i-j|\leq \varepsilon, \\
\frac{K_s}{2}(r- r_1)^2  &i \in \mathcal{L}_\varepsilon(\Omega^-),|i-j|> \varepsilon, \\
 \frac{K_f}{2}(r-\lambda_\varepsilon)^2 &i \in \mathcal{L}_\varepsilon(\Omega^+_h),|i-j|\leq \varepsilon,\\
 \frac{K_f}{2}(r- r_2(\lambda_\varepsilon))^2  &i \in \mathcal{L}_\varepsilon(\Omega^+_h),|i-j|> \varepsilon.
\end{cases}
\end{align} 
 for $K_f,K_s,\lambda_\varepsilon >0$, with $r_1:=(3k_\varepsilon^2-3k_\varepsilon+1)^{1/2}$ and $r_2(\lambda_\varepsilon):=(3k_\varepsilon^2-3(2-\lambda_\varepsilon)k_\varepsilon+3(5/4-\lambda_\varepsilon/2)^2))^{1/2}$ being by \eqref{Llength} the lengths of the vectors $(L/\varepsilon-\sqrt{3},0) +( \frac{1}{2}\sqrt{3},\pm\frac{1}{2})$ and $(L/\varepsilon-\sqrt{3},0) +\lambda_\varepsilon( \frac{1}{2}\sqrt{3},\pm\frac{1}{2})$, respectively. 
The constant $\varepsilon\lambda_\varepsilon, \varepsilon>0$ describe the \textit{equilibrium length of the film} and the \textit{equilibrium length of the substrate} respectively, $K_f,K_s$ describe the \textit{elastic constant} of the film and the substrate and $\gamma_f, \gamma_s$ describe the vapor-film and vapor-substrate \textit{surface tension} respectively. 
\medskip

 \textbf{One reference frame.}  Let us denote the elastic contribution related to film atoms by
\begin{align}\label{Elasticrewrittenfilm}
E^{el,film}_\varepsilon(y,h):=\sum_{ i \in \mathcal{L}_\varepsilon(\Omega_h^+)} \sum_{j \in \mathcal{N}_\varepsilon(i)} \varepsilon V_{i,j}^\varepsilon\left(\frac{|y(i)-y(j)|}{\varepsilon} \right),
\end{align} 
and the one related to the substrate by $E^{el,sub}_\varepsilon(y,h):= E^{el}_\varepsilon(y,h)-E^{el,film}_\varepsilon(y,h)$. We observe that 
\begin{align*}
E^{el,film}_\varepsilon(\bar{y},h) = \min_{y} E^{el,film}_\varepsilon(y,h) = 0,
\end{align*}
and 
\begin{align*}
E^{el,sub}_\varepsilon(\tilde{y},h) = \min_{y} E^{el,sub}_\varepsilon(y,h) = 0,
\end{align*}
for any  $\bar{y}_i:= \lambda_\varepsilon \bar{R} i  + \bar{z} $ and $\tilde{y}_i:= \tilde{R} i  + \tilde{z} $ with $\bar{R},\tilde{R} \in SO(s)$ and  $\bar{z}, \tilde{z} \in \mathbb{R}^2$.
This means that we are considering a \textit{reference frame} with respect to the substrate and not with respect to the film. Indeed, we have that $E^{el,film}_\varepsilon(\tilde{y},h)>0$.

\medskip

 \textbf{Relation to atomistic interactions.}   We now observe that under certain assumptions below described  the energies $E_\varepsilon$ can be justified starting from \textit{renormalized interatomic energies}, here denoted by $\mathcal{E}_\varepsilon$, where the energy contribution related to the bonding of film and substrate particles is characterized by interatomic potentials, for example of Lennard-Jones type (see Figure \ref{Fig.Lennard Jones}).  
We introduce $\mathcal{E}_\varepsilon$  as the energy defined for every discrete profile $h : \frac{\sqrt{3}}{2}\varepsilon(\mathbb{Z}+\frac{1}{2}) \cap S^1_L \to \mathbb{R}_+$ and discrete deformation  $y : \mathcal{L}_\varepsilon(\Omega_h) \to \mathbb{R}^2$ by
\begin{align}\label{auxiliary_energy}
\mathcal{E}_\varepsilon(y,h)&=  \sum_{i \in \mathcal{L}_\varepsilon(\Omega_h^+)} \sum_{j \in \mathcal{N}_\varepsilon(i)}
\varepsilon V_{f}^\varepsilon\left(\frac{|y_i-y_j|}{\varepsilon}\right) +\sum_{i,j \in \mathcal{L}_\varepsilon(\Omega^-)}\sum_{j \in \mathcal{N}_\varepsilon(i)}\varepsilon V_{s}\left(\frac{|y_i-y_j|}{\varepsilon}\right), 
\end{align}
where $V_{s}$ and $V_{f}^\varepsilon$  are phenomenological potentials from $\mathbb{R}^+$ to  $\mathbb{R}\cup\{\infty\}$ describing the interactions between substrate- and film-atoms. More precisely we assume the following properties  on $V_{s}$ and $V_{f}^\varepsilon$: 
\begin{itemize}
\item[(i)] $V_{s},V_{f}^\varepsilon \in C^2((0,+\infty))$,
\item[(ii)] $\displaystyle V_{s}(1)=-\gamma_s = \min_{r>0}V_{s}(r)$, $\displaystyle V_{f}^\varepsilon(\lambda_\varepsilon)=-\gamma_f = \min_{r>0}V_{f}^\varepsilon(r)$,
\item[(iii)] $(V_{s})''(1)= K_s$, $(V_{f}^\varepsilon)''(\lambda_\varepsilon)= K_f$.
\end{itemize}
 Notice that interactions between not nearest neighbors are not considered in \eqref{auxiliary_energy}. This can be considered as an approximation when the decay of the two potentials  $V_{s}$ and $V_{f}^\varepsilon$ is fast enough. Furthermore, we observe also that under the volume constraint on the film atoms, i.e., $||h||_{L^1}= V_\varepsilon\in\mathbb{R}$,  $\#\mathcal{L}_\varepsilon(\Omega_h^+)$ and $\#\mathcal{L}_\varepsilon(\Omega^-)$ are constants independent of $h$ and $y$. Therefore, the  quantity
$$
m_\varepsilon:=-\varepsilon(6\gamma_s \#\mathcal{L}_\varepsilon(\Omega^-) + 6\gamma_f \#\mathcal{L}_\varepsilon(\Omega_h^+),   
$$
which corresponds to the (theoretical)  bonding energy of a configuration where all particles are in equilibrium (have 6 nearest neighbors), is a real value. We observe that 
\begin{align}
\mathcal{E}_\varepsilon(y,h)-m_\varepsilon&=  \sum_{i,j \in \mathcal{L}_\varepsilon(\Omega^-)}\sum_{j \in \mathcal{N}_\varepsilon(i)}\varepsilon \left(V_{s}\left(\frac{|y_i-y_j|}{\varepsilon}\right)+ \gamma_s\right)+  \gamma_s \sum_{i \in \mathcal{L}_\varepsilon(\Omega^-)} \varepsilon(6-\#\mathcal{N}_\varepsilon(i)) \nonumber\\
&\quad+ \sum_{i \in \mathcal{L}_\varepsilon(\Omega_h^+)} \sum_{j \in \mathcal{N}_\varepsilon(i)}
\varepsilon \left(V_{f}^\varepsilon\left(\frac{|y_i-y_j|}{\varepsilon}\right)+\gamma_f\right) +  \gamma_f \sum_{i \in \mathcal{L}_\varepsilon(\Omega_h^+)} \varepsilon(6-\#\mathcal{N}_\varepsilon(i)). \label{auxiliar_energy_2} 
\end{align}

By (i)-(iii) and under the assumptions that 
\begin{itemize}
\item[1.] the discrete deformations $y$ are small, i.e., 
\begin{equation}\label{ineq: small deformation}
 \left|\left|\frac{y(i)-y(j)}{\varepsilon}\right|-1\right|  \leq C\sqrt{\varepsilon}
\end{equation}
for every deformation $y$ and  for all  $i,j \in \mathcal{L}_\varepsilon(\Omega_h)$ with  $j \in \mathcal{N}_\varepsilon(i)$,
\item[2.]  the lattice mismatches are small, i.e., $|\lambda_\varepsilon - 1| \leq C\sqrt{\varepsilon}$,
\end{itemize}
we can Taylor expand $V_{s}$ and $V_{f}^\varepsilon$ around their respective minimum points obtaining 
\begin{align}
 V_{s}\left(\left|\frac{y_i-y_j}{\varepsilon}\right|\right)   &=  V_{s}(1)+ V'_{s}(1)\left(\left|\frac{y_i-y_j}{\varepsilon}\right|-1\right) + \frac{1}{2}
 V''_{s}(1)\left(\left|\frac{y_i-y_j}{\varepsilon}\right|-1\right)^2 + o(\varepsilon) \nonumber\\
 &=-\gamma_s + \frac{1}{2}K_s\left(\left|\frac{y_i-y_j}{\varepsilon}\right|-1\right)^2 + o(\varepsilon), \label{taylor_Vs}
\end{align}
and, similarly,
\begin{align}\label{taylor_Vf}
V_{f}\left(\left|\frac{y_i-y_j}{\varepsilon}\right|\right)  = - \gamma_f + \frac{1}{2}K_f\left(\left|\frac{y_i-y_j}{\varepsilon}\right|-\lambda_\varepsilon\right)^2 + o(\varepsilon).
\end{align}

Therefore, from \eqref{auxiliar_energy_2}, \eqref{taylor_Vs}, and \eqref{taylor_Vf} it follows that 
\begin{align}
\mathcal{E}_\varepsilon(y,h) - m_\varepsilon &=  \frac{K_s}{2}\sum_{i,j \in \mathcal{L}_\varepsilon(\Omega^-)}\sum_{j \in \mathcal{N}_\varepsilon(i)}\varepsilon\left(\left|\frac{y_i-y_j}{\varepsilon}\right|-1\right)^2+  \gamma_s  \sum_{i \in \mathcal{L}_\varepsilon(\Omega^-)} \varepsilon(6-\#\mathcal{N}_\varepsilon(i)) \nonumber\\
&\quad+ \frac{K_f}{2} \sum_{i \in \mathcal{L}_\varepsilon(\Omega_h^+)} \sum_{j \in \mathcal{N}_\varepsilon(i)}
\varepsilon\left(\left|\frac{y_i-y_j}{\varepsilon}\right|-\lambda_\varepsilon\right)^2   +  \gamma_f  \sum_{i \in \mathcal{L}_\varepsilon(\Omega_h^+)} \varepsilon(6-\#\mathcal{N}_\varepsilon(i)) + o(1), \label{auxiliar_energy_3} 
\end{align}
where we also used that $\#\mathcal{L}_\varepsilon(\Omega_h) \leq C \varepsilon^{-2}$.
In view of \eqref{auxiliar_energy_3} we can say that, at least at a formal level, the energies $E_\varepsilon$ of the discrete model introduced  in this paper are justified starting from atomistic interactions for $\varepsilon>0$ small enough.

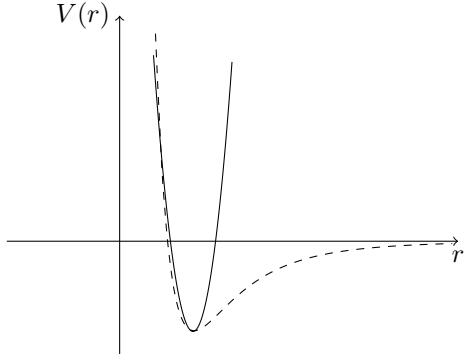
\begin{figure}[htp]
\centering
\begin{tikzpicture}[scale=3]
\draw[->](0.25,0)--++(2,0);
\draw(2.25,0) node[anchor=north]{$r$};
\draw[->](0.75,-0.5)--++(0,1.5);
\draw(0.75,1) node[anchor=east]{$V(r)$};
\draw[smooth,samples=150,domain=0.75:11,dashed] plot ({abs(\x)^(1/3)},{1/(\x*\x*\x*\x)-2^(1/3)/(\x*\x)});
\draw[smooth,samples=150,domain=.9:1.25] plot ({\x},{40*(\x-1.075)^2-.4});
\end{tikzpicture}
\caption{The Lennard--Jones potential and the interaction potential of harmonic springs}
\label{Fig.Lennard Jones}
\end{figure}


\medskip
 \textbf{Rescaled displacements.}    Another important quantity will be the rescaled  \emph{discrete displacements} $u$ associated to a deformation $y$.  For a rotation $R \in SO(2)$ and $b \in \mathbb{R}^2$ we define the rescaled displacement $u:\mathcal{L}_\varepsilon(\Omega_h)\to\mathbb{R}^2$ (omitting the dependence on $\varepsilon$)  associated  to $y$ by 
\begin{align}\label{urescaled}
u(x) = \frac{y(x)-(Rx+b)}{\sqrt{\varepsilon}}.
\end{align}

\medskip

\medskip
 \textbf{The configurational space X.} 
In the following we refer to   triples $(y,u,h)$ where $h$ is a discrete profile defined in , $y$ is a orientation-preserving discrete deformation, $u$ is a discrete displacement associated to $y$, as \emph{discrete triples} and we denote the space of discrete triples by $X_d$.  
In order to perform a discrete-to-continuum analysis it is convenient to embed $X_d$ in the larger  configurational  space 
\begin{equation}\label{Function Space}
\begin{split}
X:=\left\{(y,u,h) : h \in AP(S^1_L), \Omega_h \text{ given by } (\ref{Defintion Omega}),y,u \in L^2_{\mathrm{loc}}(\Omega_h;\mathbb{R}^2) \right\}.
\end{split}
\end{equation}

To do that we identify each discrete profile $h$ with the lower semicontinuous interpolation given by \eqref{h Interpolation}, the discrete deformations   $y : \mathcal{L}_\varepsilon(\Omega_h) \to \mathbb{R}^2$ with the piecewise affine interpolation defined in the following  way: For every  $T=\{i_1,i_2,i_3\} \in \mathcal{T}_\varepsilon$ we set
\begin{align}\label{y interpolate}
y(x) = \sum_{k=1}^3 \lambda_k y(i_k)
\end{align}
for every  $x$ that can be written as
$$x = \sum_{k=1}^3 \lambda_k i_k$$ 
for some $\lambda_k\geq0$, $k=1,2,3$, such that 
$$
 \sum_k \lambda_k =1, 
$$
and $y(x) =0$ if there is no (unique) triple $i_1,i_2,i_3 \in \mathcal{L}_\varepsilon(\Omega_h)$ with $|i_k-i_j|\leq \varepsilon, j\neq k$ and $x \in \mathrm{conv}(i_1,i_2,i_3)$.  Note that, for neighbors though the boundaries $x_1=0$ and $x_1=L$, if   $\{i_1,i_2,i_3\}$ such that $|i_1\pm Le_1 -i_k|\leq \varepsilon,k=2,3$ we define $y$ in the same way as in (\ref{y interpolate}) with $x \in \mathrm{conv}(i_1,\mp Le_1 +i_2,\mp Le_1 +i_3)$,  while  if $\{i_1,i_2,i_3\}$ are such that $|i_1\pm Le_1 -i_2|,|i_1-i_3| \leq \varepsilon$ we define $y$ in the same way as in (\ref{y interpolate}) with $x \in \mathrm{conv}(i_1,\mp Le_1 +i_2,i_3)$. This procedure is well defined up to a set of lebesgue-measure $0$ and we can therefore interpret $y \in L^2_{\mathrm{loc}}(\Omega_h)$. Notice that in this way $y(0,x_2) = y(L,x_2)$ for $\mathcal{H}^1$-a.e. $x_2 \in (-R,h(0))$.

 Similarly,  every discrete displacement  $u:\mathcal{L}_\varepsilon(\Omega_h)\to\mathbb{R}^2$ can be interpreted as an element of $L^2_{\mathrm{loc}}(\Omega_h)$ by identifying it with its piecewise constant interpolation as for $y$ above.   We write that $X_d\subset X$. 

 \begin{definition}[Convergence  in $X$] \label{Definition Convergence} We consider in $X$ the topology $\tau_X$ related to the following definition  of convergence: 
 A sequence $(y_\varepsilon,u_\varepsilon, h_\varepsilon) \subset X $ is said to converge to $(y,u,h) \in X$ in $X$, and we write $(y_\varepsilon,u_\varepsilon,h_\varepsilon) \to (y,u,h)$, if 
\begin{itemize}
\item[(i)] the sets $Q\setminus \Omega_{h_\varepsilon} $ converge to $ Q\setminus \Omega_h$ with respect to the Hausdorff-distance;
\item[(ii)] $y_\varepsilon \to y$ in $L^2_{\mathrm{loc}}(\Omega_h;\mathbb{R}^2)$;
\item[(iii)] $u_\varepsilon \to u$ in $L^2_{\mathrm{loc}}(\Omega_h;\mathbb{R}^2)$.
\end{itemize}
\end{definition}


 We notice that the condition (ii) and  (iii) in Definition \ref{Definition Convergence} are well defined, since  by i)  it follows that for $\Omega' \subset\subset \Omega_h$ we have that $\Omega' \subset\subset \Omega_{h_\varepsilon}$ for $\varepsilon>0$ sufficiently small. Furthermore, observe that that this convergence is metrizable  with a metric that we denote by $d_X$.

\medskip

\medskip
 \textbf{The extended energy.}  Fix $V_\varepsilon >0$ such that there exist a discrete profile $h$ such that $||h||_{L^1(S^1_L)}=V_\varepsilon$.  Now we extend the energy $E_\varepsilon$ defined in \eqref{Definition Energy} for discrete profiles $h$ and deformations $y$ to the whole space $X$ by extending it to $+\infty$ outside  $X_d$.  More precisely, we  write (with slight abuse of notation) $E_\varepsilon : X \to [0,+\infty]$ given by
\begin{align}\label{Definition Energy2}
E_\varepsilon(y,u,h) = \begin{cases}E_\varepsilon(y,h)  & \text{if $(y,u,h)\in X_d$ and $||h||_1 =V_\varepsilon $,}\\
+\infty &\text{otherwise.} 
\end{cases}
\end{align}

\medskip
\subsection{ The limiting model}
 In this subsection we introduce the continuum  limiting model. To this end let us assume that the \emph{discrete lattice mismatch}  $\delta_\varepsilon :=(\lambda_\varepsilon -1)$ satisfy
\begin{align}\label{Mismatch}
\lim_{\varepsilon \to 0 } \varepsilon^{-\frac{1}{2}}\delta_\varepsilon= \delta \in \mathbb{R}.
\end{align}
In the following we refer to $\delta \in \mathbb{R}$ as the \textit{lattice mismatch}.
 For a triple $(y,u,h) \in X$  we define the limit elastic energy by
 \begin{align}
 E^{el}(y,u,h) =\int_{\Omega_h} W_y(x,Eu(x))\mathrm{d}x,
 \end{align}
 where $Eu=\frac{1}{2}(\nabla u + \nabla u^T)$ is the symmetric part of the gradient of $u$ and $W_y : \Omega_h \times \mathbb{R}^{2\times2 } \to [0,+\infty]$ is given by
 \begin{align*}
W_y(x,A) = \begin{cases} \frac{8 K_f}{\sqrt{3}}\left(2|A-\delta E y|^2 + (\mathrm{trace}(A-\delta E y))^2\right)& x \in \Omega_h^+, A \in \mathbb{R}^{2\times 2}, \\
 \frac{8 K_s}{\sqrt{3}}\left(2|A|^2 + (\mathrm{trace}(A))^2\right)&x \in \Omega^-, A \in \mathbb{R}^{2\times 2}.
\end{cases}
 \end{align*}
The limiting surface energy $E^S : AP(S^1_L) \to [0,+\infty]$ is defined by
 \begin{align} \label{ESenergy}
 \begin{split}
 E^S(h) &= \gamma_f\left(\int_{\partial \Omega_h \cap \Omega_h^{1/2}\cap Q^+} \varphi(\nu)\mathrm{d}\mathcal{H}^1+2\int_{\partial \Omega_h \cap \Omega_h^{1}\cap Q} \varphi(\nu)\mathrm{d}\mathcal{H}^1\right)\\&\qquad\qquad\qquad+\gamma_s\wedge\gamma_f\int_{\partial \Omega_h \cap \Omega_h^{1/2}\cap \{x_2=0\}} \varphi(\nu)\mathrm{d}\mathcal{H}^1
 \end{split}
 \end{align}
 where the surface tension $\varphi : \mathbb{R}^2 \to [0,+\infty)$ is defined by
\begin{align}\label{Definitionphi}
 \varphi(\nu) =  \frac{2}{3}\sqrt{3}\left( |\nu_2|+\frac{1}{2}|\sqrt{3}\nu_1-\nu_2|+ \frac{1}{2}|\sqrt{3}\nu_1+\nu_2|  \right).  
\end{align}

Here $\nu(x) \in S^1$ is defined as $\tau^\perp(x)= (-\tau_2(x),\tau_1(x))$ where $\tau(x)=(\tau_1(x),\tau_2(x))$ is the unit tangent vector to the set $\partial \Omega_h$ at the point $x \in \partial \Omega_h$. Since $\partial \Omega_h$ is connected and $\mathcal{H}^1(\partial \Omega_h) < +\infty$ due to \cite{falconer1986geometry} Theorem 3.8 the tangent $\tau$ is well-defined for $\mathcal{H}^1$-a.e. $x \in \partial \Omega_h$. For points $x \in \partial \Omega_h \cap \Omega_h^{1/2}$ the vector $\nu(x)$ is the unit inner normal to the set $\Omega_h$ whenever it exists. The function $x \mapsto \nu(x)$ is Borel-measurable so that for every continuous function $\varphi :\mathbb{R}^2 \to [0,+\infty)$ the functional (\ref{ESenergy}) is well defined.  We also observe that  a discontinuity for the surface tension in \eqref{ESenergy} (apart from the cuts in the graph of $h$) may occur when  $\gamma_s>0$, representing the surface tension between the substrate and the  vapor,  is lower than the surface tension $\gamma_f$ between the film and the  vapor.  If  $\gamma_f < \gamma_s$, the surface energy density is no longer discontinuous and in fact is equal to $\gamma_f \varphi(\nu)$. 

\medskip
We define the limit energy $E:X \to [0,+\infty]$  for every $V >0$ such that $V_\varepsilon \to V >0$   by
\begin{align}\label{Definition Limit Energy}
E(y,u,h) = \begin{cases} E^{el}(y,u,h) + E^S(h)&  \text{if $(y,u,h)\in X_c$ and $||h||_1 = V$,}\\ 
+\infty  &\text{if $(y,u,h)\in X\setminus X_c$,} 
\end{cases}
\end{align} 
where
\begin{align*}
X_c:=\{(y,u,h)\in X\,:\, &\text{$u \in H^1_{\mathrm{loc}}(\Omega_h;\mathbb{R}^2)$,  $y = Rx +b$  for some  $(R,b)\in SO(2)\times \mathbb{R}^2$,}\\
& \text{$h(0)=h(L)$,  and $u(0,x_2)=u(L,x_2)$ for $\mathcal{H}^1\text{-a.e. } x_2 \in (-R,h(0))$} \}.
\end{align*}


\section{Compactness} 
 In this section we show that $\tau_X$ is a good choice of topology, since sequences with equi-bounded energies are pre-compact in $X$ with respect to this topology (see  Proposition \ref{Compactness Proposition}). The main tool is represented by the rigidity estimate proved in \cite{friesecke2002theorem} that we recall here for the reader's convenience.

\medskip

\begin{theorem}[\cite{friesecke2002theorem} Theorem 3.1] \label{TheoremRigidity} Let $N\geq 2$ and let $1<p<\infty$. Suppose that $U\subset \mathbb{R}^N$ is a bounded Lipschitz domain. Then there exists a constant $C=C(U)$ such that for every $u \in W^{1,p}(U)$, there exists a constant matrix $R \in SO(N)$ such that
\begin{align*}
||\nabla u - R||_{L^p(U;\mathbb{R}^{N\times N})} \leq C(U)||\mathrm{dist}(\nabla u, SO(N))||_{L^p(U)}.
\end{align*}
The constant $C(U)$ is invariant under dilation or translation of $U$.
\end{theorem}

In order to apply Theorem \ref{TheoremRigidity} we regroup the elastic energy as the sum of cell-energies on the triangular faces of the lattice.  We  denote  the family of   triangles in $\mathcal{L}_\varepsilon$ by $\mathcal{T}_\varepsilon$.

and the \emph{cell energy} of such a triangle $T =\{i_1,i_2,i_3\}\in \mathcal{T}_\varepsilon$ 
by 
\begin{align*}
W_{\varepsilon,cell}(F,T) =  \underset{ k\neq j}{\sum_{k,j=1}^3} \tilde{V}_{i_k,i_j}^\varepsilon(F(i_k-i_j)), 
\end{align*}
where $\tilde{V}_{i,j}^\varepsilon : \mathbb{R}^2 \to \mathbb{R}$ is given by
\begin{align*}
\tilde{V}_{i,j}^\varepsilon(\xi) = \begin{cases}\frac{1}{2}V_{i,j}^\varepsilon(|\xi|) &i,j \in \mathcal{L}_\varepsilon(\Omega_h),\xi \in \mathbb{R}^2,\\ 0 &\text{if } i \text{ or } j \notin \mathcal{L}_\varepsilon(\Omega_h),\xi \in \mathbb{R}^2
\end{cases}
\end{align*}
with  $V_{i,j}^\varepsilon$  defined by (\ref{DefinitionVi}). For any $T=\{i_1,i_2,i_3\} \in \mathcal{T}_\varepsilon$ we set $\hat{x}_T = \frac{1}{3}(i_1+i_2+i_3)$.  Note that now by this definition and (\ref{Elasticrewritten}) we have that
\begin{align*}
E^{el}_\varepsilon(y,h) = \sum_{i \in \mathcal{T}_\varepsilon}\varepsilon W_{\varepsilon,cell}\left(\nabla y(\hat{x}_T),T\right).
\end{align*}
Note that $\nabla y$ is the gradient of its piecewise affine interpolation given by (\ref{y interpolate}). 
 In order to show that  cell energies $W_{\varepsilon,cell}(\nabla y(\hat{x}_T),T)$ control the distance of $\nabla y(\hat{x}_T)$ (see  Proposition  \ref{PropositionBound}) from the set of rotations we need  the following Lemma.

\begin{lemma}\label{LemmaDistance} There exists a constant $C>0$ such that for all $\lambda>0$ and all $F \in \mathbb{R}^{2\times 2}$ with  $\det F \geq 0$ there holds
\begin{align*}
\mathrm{dist}^2(F,\lambda SO(2)) \leq C W_{\lambda}(F),
\end{align*}
 where $W_{\lambda} : \mathbb{R}^{2\times 2} \to [0,+\infty]$  is defined by
\begin{align*}
W_{\lambda}(F) := \begin{cases}(|Fe|-\lambda)^2 + (|Fv|-\lambda)^2+(|F(v-e)|-\lambda)^2 &\mathrm{det}F \geq 0,\\
+\infty &\text{otherwise,} 
\end{cases}
\end{align*}
with $e=(1,0)$ and  $v = \frac{1}{2}(1,\sqrt{3})$.  
\end{lemma}
\begin{proof} The statement follows by checking that $W_{\lambda}$ satisfies
\begin{itemize}
\item[i)] $W_{\lambda}(RF)=W_{\lambda}(F)$ for all $F \in \mathbb{R}^{2\times 2}, R \in SO(2)$.
\item[ii)] $\{W_{\lambda}=0\} \cap \{\det F \geq 0\} = \lambda SO(2)$
\item[iii)] $W_\lambda \in C^2$ in a neighbourhood of $\lambda SO(2)$ and $D^2W_\lambda(\lambda Id)$ is positive definite on the orthogonal complement of the subspace spanned by infinitesimal rotations, that is $F \mapsto A F$, $A^T = -A$.
\item[iv)] $ \displaystyle
\lim_{F \to +\infty} \frac{W_\lambda(F)}{|F|^2} >0
$.
\end{itemize}
The rest of the proof is similar to the one given in \cite{schmidt2009derivation} for Lemma 3.2.
\end{proof}
 The following proposition will be crucial to prove the compactness result \ref{Compactness Proposition}.  

\begin{proposition} \label{PropositionBound} Let $y : \mathcal{L}_\varepsilon(\Omega_h) \to \mathbb{R}^2$ be orientation preserving and let $T \in T_\varepsilon$ be such that $T=\{i_1,i_2,i_3\}$ with $i_1,i_2,i_3 \in \mathcal{L}_\varepsilon(\Omega_h)$, then
\begin{align}
W_{\varepsilon,cell}(\nabla y(\hat{x}_T),T) \geq c\left(\mathrm{dist}^2(\nabla y(\hat{x}_T),SO(2)) - \varepsilon\right).
\end{align}
\end{proposition}
\begin{proof} Since $i_1,i_2,i_3 \in \mathcal{L}_\varepsilon(\Omega_h)$ we have that $V_{i,j}^\varepsilon(\xi) = V_i^\varepsilon(|\xi|)$. By convexity there holds
\begin{align}\label{ineq : conv}
(|\xi|-\lambda)^2= (|\xi|-1 +1-\lambda)^2 \geq c (|\xi| -1)^2 - c(1-\lambda)^2=c (|\xi| -1)^2- c\varepsilon.
\end{align}
Now the claim follows by applying Lemma \ref{LemmaDistance} and (\ref{ineq : conv}) to $W_{\varepsilon,cell}(F,T)$ to obtain
\begin{align*}
W_{\varepsilon,cell}(\nabla y(\hat{x}_T),T) &=   \underset{ k\neq j}{\sum_{k,j=1}^3} \tilde{V}_{i_k,i_j}^\varepsilon(\nabla y(\hat{x}_T)(i_k-i_j)) \geq c  \underset{ k\neq j}{\sum_{k,j=1}^3} V_{i_k}^\varepsilon(|\nabla y(\hat{x}_T)(i_k-i_j)|) \\&\geq c W_1(\nabla y(\hat{x}_T)) - c\varepsilon \geq c\left(\mathrm{dist}^2(\nabla y(\hat{x}_T),SO(2)) - \varepsilon\right).
\end{align*}
This concludes the proof.
\end{proof} 

 We now state the compactness results which is based to the rigidity estimate \ref{TheoremRigidity}.  Since we have domains with varying boundary profile it is not possible, however, to apply the rigidity estimate to the whole domain. This is also reflected by the topology that we have chosen. We need to prove that  rigid motions  around which we linearize can be chosen independently of the compact set $\Omega'\subset\subset \Omega_h$.

\begin{proposition}[Compactness]\label{Compactness Proposition} Let $\lambda_\varepsilon \to 1$ be such that
\begin{align*}
\sup_{\varepsilon >0} \varepsilon^{-\frac{1}{2}}|1-\lambda_\varepsilon| <+\infty
\end{align*} 
and let $(y_\varepsilon,u_\varepsilon,h_\varepsilon) \subset X$ be such that
\begin{align*}
\sup_{\varepsilon >0 } E_\varepsilon(y_\varepsilon,u_\varepsilon,h_\varepsilon) < +\infty 
\end{align*}
Then there exists $(R_\varepsilon,b_\varepsilon) \subset SO(2) \times \mathbb{R}^2$, a subsequence (not relabelled) and $(y,u,h) \in X$,$R\in SO(2),b\in \mathbb{R}^2,||h||_1 = V, u(0,x_2) =u(L,x_2)$  for $\mathcal{H}^1$-a.e. $x_2 \in (-R,h(0))$  such that $R_\varepsilon \to R$,$b_\varepsilon \to b$, $y=Rx+b$ and
\begin{align*}
 \left(y_\varepsilon,\frac{y_\varepsilon-(R_\varepsilon x+b_\varepsilon)}{\sqrt{\varepsilon}},h_\varepsilon\right) \to (y,u,h),
\end{align*}
 with respect to the convergence given in Definition \ref{Definition Convergence}. Moreover we have that 
 \begin{align}\label{Compactnessu}
 \frac{y_\varepsilon-(R_\varepsilon x+b_\varepsilon)}{\sqrt{\varepsilon}}\rightharpoonup u, \text{ in } H^1_{\mathrm{loc}}(\Omega_h).
 \end{align}
\end{proposition}
\begin{proof} Let $\delta_\varepsilon=(1-\lambda_\varepsilon) \to 0, (y_\varepsilon,u_\varepsilon,h_\varepsilon) \subset X$ satisfy the assumptions of Proposition \ref{Compactness Proposition}, that is there exists $0<C<+\infty$ such that
\begin{align*}
\sup_{\varepsilon >0} \varepsilon^{-\frac{1}{2}}|\delta_\varepsilon| \leq C \text{ and } \sup_{\varepsilon >0} E_\varepsilon(y_\varepsilon,u_\varepsilon,h_\varepsilon) \leq C.
\end{align*}
We first prove i) of Definition \ref{Definition Convergence}.
One can check that for piecewise constant functions it holds
\begin{align*}
|Dh_\varepsilon|(S^1_L)=\mathrm{Var} h_\varepsilon = \sum_{i \in \frac{\sqrt{3}}{2}\varepsilon(\mathbb{Z}+\frac{1}{2}) \cap S^1_L} \left|h_\varepsilon(i+\frac{1}{2}\sqrt{3}\varepsilon) - h_\varepsilon(i)\right|.
\end{align*}
Fix $i \in \frac{\sqrt{3}}{2}\varepsilon(\mathbb{Z}+\frac{1}{2}) \cap S^1_L$, we have that
\begin{align}\label{ibound}
\underset{j_1\in  \{i,i+\frac{1}{2}\sqrt{3}\varepsilon\}}{\sum_{(j_1,j_2) \in \mathcal{L}_\varepsilon(\Omega_h^+)}} \varepsilon(6 - \#\mathcal{N}_\varepsilon((j_1,j_2)) \geq \left|h_\varepsilon(i+\frac{1}{2}\sqrt{3}\varepsilon) - h_\varepsilon(i)\right|,
\end{align}
since for all $j_2 \in \{\min\{h_\varepsilon(i),h_\varepsilon\left(i+\frac{1}{2}\sqrt{3}\varepsilon\right)\},\ldots, \max\{h_\varepsilon(i),h_\varepsilon\left(i+\frac{1}{2}\sqrt{3}\varepsilon\right)\}\}$ we have that either $\#\mathcal{N}_\varepsilon((i,j_2)<6$ or $\#\mathcal{N}_\varepsilon((i+\frac{1}{2}\sqrt{3}\varepsilon,j_2)<6$. Summing over $i \in \frac{\sqrt{3}}{2}\varepsilon(\mathbb{Z} +\frac{1}{2})\cap S^1_L$ and using (\ref{ibound}) we obtain
\begin{align*}
\mathrm{Var}h_\varepsilon \leq C( E_\varepsilon(y_\varepsilon,h_\varepsilon) +1) \leq C <+\infty.
\end{align*}
 Moreover, since there exist $\{x_\varepsilon\}_\varepsilon \subset [0,L]$ such that
\begin{align*}
\sup_{\varepsilon>0} h_\varepsilon(x_\varepsilon) \leq  \sup_\varepsilon\fint_0^L h_\varepsilon(x) \mathrm{d}x \leq C 
\end{align*} 
we have that
$
h_\varepsilon(x) \leq h_\varepsilon(x_\varepsilon) + \mathrm{Var}h_\varepsilon \leq C 
$ for all $x \in [0,L]$. Now for all $\varepsilon >0$ we have that $\Omega_{h_\varepsilon} \subset \{(x_1,x_2) : 0 < x_1 < L,-R< x_2 < l\}$ for some $l >0$. Hence the compactness of the sets $Q\setminus \Omega_{h_\varepsilon}$ is equivalent to the compactness of the equibounded sets $\{(x_1,x_2) : 0<x_1<L,-R < x_2 \leq l\} \setminus \Omega_{h_\varepsilon}$, which follows from the Blaschke Compactness Theorem (cf. Theorem 6.1 in  \cite{ambrosio2000functions}).  Thus we may assume that, up to subsequecens (not relabelled) $Q\setminus \Omega_{h_\varepsilon}$ converges in the Hausdorff-metric to a set $Q\setminus \Omega$. Next we identify $\Omega$ with $\Omega_h$, where
\begin{align}\label{hdefinition}
h(x) = \inf\left\{\liminf_{\varepsilon \to 0} h_\varepsilon(x_\varepsilon) : x_\varepsilon \to x\right\}.
\end{align} 
Note that a sequence $K_\varepsilon$ of compact sets contained in a compact set $U$ converge to $K$ in the Hausdorff-metric if and only if the following hold true
\begin{itemize}
\item[i)] for all $x \in K$, there exists $x_\varepsilon \to x$ such that $x_\varepsilon \in K_\varepsilon$,
\item[ii)] for all $\{x_\varepsilon\} $ such that $x_\varepsilon \in K_\varepsilon$ and $x_\varepsilon \to x$ we have that $x \in K$.
\end{itemize}
Let $x = (x_1,x_2) \in Q\setminus \Omega$, by i) there exist $x_\varepsilon = (x_1^\varepsilon,x_2^\varepsilon) \in Q\setminus \Omega_{h_\varepsilon}$ and $x_\varepsilon \to x$. Since $x_\varepsilon \in Q\setminus \Omega_{h_\varepsilon} $ and we have that $x_\varepsilon \to x$ we obtain
\begin{align*}
h(x_1) \leq \liminf_{\varepsilon \to 0} h_\varepsilon(x_\varepsilon^1)\leq \liminf_{\varepsilon \to 0} x_\varepsilon^2 = x_2.
\end{align*}
Hence $x \in Q\setminus \Omega_h$ which implies $Q\setminus \Omega \subset Q\setminus \Omega_h$. Now let $x=(x_1,x_2) \in Q\setminus \Omega_h$. We have
\begin{align*}
h(x_1) = \inf\{\liminf_{\varepsilon \to 0} h_\varepsilon(y_\varepsilon) : y_\varepsilon \to x_1\} \leq x_2.
\end{align*}
Let $x_\varepsilon=(x_\varepsilon^1,x_\varepsilon^2)$ be such that $h_\varepsilon(x_\varepsilon^1) \to h(x_1)$ and $x_\varepsilon^1 \to x_1$ and $x_\varepsilon^2 = \max\{x_2,h_\varepsilon(x_\varepsilon^1)\}$. We have that $x_\varepsilon \in Q\setminus \Omega_{h_\varepsilon}$ and $x_\varepsilon \to x$. By ii) it follows that $x \in Q\setminus \Omega$ which implies $Q\setminus \Omega_h \subset Q\setminus \Omega$. Finally we need to show that $h $ is lower semicontinuous and $\mathrm{Var}h <+\infty$. We have that (up to a subsequence) $Q\setminus \Omega_{h_\varepsilon} \to Q\setminus \Omega_h$ with respect to the Hausdorff-distance with $h$ given by (\ref{hdefinition}).

 By its definition it is easy to check that $h$ is a lower-semicontinuous function. 
Due to \cite{fonseca2007equilibrium} Lemma 2.5 we have that $h_\varepsilon \to h$ in $L^1(S^1_L)$ and therefore
\begin{align*}
V = \lim_{\varepsilon \to 0} V_\varepsilon = \lim_{\varepsilon \to 0} ||h_\varepsilon||_{L^1(S^1_L)} = ||h||_{L^1(S^1_L)}.
\end{align*}
The constraint $||h||_{L^1(S^1_L)} = V$ is therefore satisfied. By Blaschke's Compactness Theorem we have that there exists $K\subset \mathbb{R}^2$ compact and a  subsequence (not relabelled) such that $\partial \Omega_{h_\varepsilon} \to K$ with respect to the Hausdorff-convergence. It can be checked that $\partial \Omega_h \subset K$. By Golab's Theorem there holds
\begin{align*}
\mathcal{H}^1(\partial \Omega_h) \leq \mathcal{H}^1(K) \leq \liminf_{\varepsilon\to 0} \mathcal{H}^1(\partial \Omega_{h_\varepsilon}).
\end{align*}
By Lemma 2.1 in \cite{fonseca2007equilibrium} we have that $\mathrm{Var}h < +\infty$ and $\mathrm{i)}$ follows.
 
 Next we prove ii) and iii). We show that there exists $\{R_\varepsilon\}_\varepsilon \subset SO(2),\{b_\varepsilon\}_\varepsilon \subset \mathbb{R}^2$ such that for any $\Omega' \subset\subset \Omega_h$ there exists a constant $C=C(\Omega')$ such that 
\begin{align}\label{Compact H1 Bound}
||u_\varepsilon||_{H^1(\Omega')} \leq C,
\end{align}
 where
$u_\varepsilon : \Omega_h \to \mathbb{R}^2$ is defined by
\begin{align*}
u_\varepsilon(x) = \frac{y_\varepsilon(x) -(R_\varepsilon x+b_\varepsilon)}{\sqrt{\varepsilon}}.
\end{align*}

 Note that each $\Omega' \subset\subset \Omega_h$ is contained in $\Omega_{h_\lambda}$ for $\lambda>0$ big enough, where $h_\lambda : S^1_L \to [0,+\infty)$ is the $\lambda$-Yosida Transform of $h$ given by
\begin{align*}
h_\lambda(x) = \inf\left\{h(y) + \lambda|x-y|, y \in S^1_L \right\}.
\end{align*}
This follows, since $h$ is lower semicontinuous and therefore $h_\lambda$ $\Gamma$-converges to $h$ (with respect to the usual distance on $\mathbb{R}$). The $\Gamma$-convergence is equivalent to the Kuratowski-convergence of the epigraphs (\cite{dal2012introduction}, Theorem 4.16) which in turn is equivalent to the Hausdorff-convergence of the sets $Q\setminus \Omega_{h_\lambda} $ to $Q\setminus \Omega_{h}$ (already noted in the proof of $\Omega = \Omega_h$ above). Moreover we translate $h_\lambda$ away from $h$ such that we are sure not touch the profile 
\begin{align*}
\tilde{h}_{\lambda}(x) = h_\lambda(x)-\frac{1}{\lambda}
\end{align*}   
In the following $C_\lambda$ (resp. $C_{\lambda,\mu}$) denotes a constant depending on $\lambda$ (resp. on $\lambda$ and $\mu$).
It still holds true that every $\Omega' \subset \Omega_{\tilde{h}_\lambda}$ for $\lambda>0$ big enough and furthermore we have that $\Omega_{\tilde{h}_\lambda}^+  \subset \Omega_h^+$. It suffices to prove the claim for 
\begin{align*}
\Omega_\lambda = \Omega_{\tilde{h}_\lambda} \cap \left(\left(0,L\right)\times \mathbb{R}\right)
\end{align*}
since for any $\Omega'\subset\subset \Omega_h$ there exists $\lambda>0$ big enough such that $\Omega'\subset \Omega_\lambda$. We now have that there exists $R_\varepsilon^\lambda \in SO(2)$ such that
\begin{align}\label{Rigidity}
\begin{split}
E_\varepsilon(y_\varepsilon,u_\varepsilon, h_\varepsilon) 
&\geq\underset{T \cap \Omega_\lambda \neq \emptyset}{\sum_{T \in \mathcal{T}_\varepsilon}}\varepsilon W_{\varepsilon,cell}(\nabla y_\varepsilon(\hat{x}_T),T)\\
&\geq C\underset{T \cap \Omega_\lambda \neq \emptyset}{\sum_{T \in \mathcal{T}_\varepsilon}}\varepsilon\left(\mathrm{dist}^2(\nabla y_\varepsilon(\hat{x}_T),SO(2))- \varepsilon\right)  \\
&\geq \frac{C}{\varepsilon}\int_{\Omega_\lambda}\mathrm{dist}^2(\nabla y_\varepsilon(x),SO(2))\mathrm{d}x - C|\Omega_h| \\
&\geq \frac{C_\lambda}{\varepsilon}\int_{\Omega_\lambda}|\nabla y_\varepsilon(x) -R^\lambda_\varepsilon|^2\mathrm{d}x-C|\Omega_h|.
\end{split}
\end{align} 
Where the first inequality is due to the fact that $W_{\varepsilon,cell} \geq 0$ and the fact the do not count such cell-energies in the second term. The second follows since for all $T \in \mathcal{T}_\varepsilon$ such that $T \cap \Omega_\lambda \neq \emptyset$ we have that $i_1,i_2,i_3 \in \mathcal{L}_\varepsilon(\Omega_h)$ and Proposition \ref{PropositionBound}.
The third inequality is due to the fact the the summation can be seen as an integration of piecewise constant function on the triangles $T \in \mathcal{T}_\varepsilon$, where $|T| \sim \varepsilon^2$ together with $\Omega_\lambda \subset \Omega_h$ and the last inequality follows due to Theorem \ref{TheoremRigidity}. 
Since $\Omega_\lambda$ is a Lipschitz set by Poincar\'es inequality we have that there exists $b_\varepsilon \in \mathbb{R}^2$ such that
\begin{align}\label{Poincare}
\int_{\Omega_\lambda}|y_\varepsilon(x) -(R^\lambda_\varepsilon x + b^\lambda_\varepsilon)|^2\mathrm{d}x\leq C_\lambda\int_{\Omega_\lambda}|\nabla y_\varepsilon(x) -R^\lambda_\varepsilon|^2\mathrm{d}x. 
\end{align}
Now fix $\mu>\lambda >0$. We have that there exist $R_\varepsilon^{\lambda},R_\varepsilon^{\mu}\in SO(2)$ such that (\ref{Rigidity}) holds true. We have
\begin{align}\label{Rotationdifference}
\begin{split}
|R_\varepsilon^\lambda-R_\varepsilon^\mu|^2 &\leq C_{\lambda,\mu}\left(\fint_{\Omega_\lambda} |R_\varepsilon^\lambda-\nabla y(x)|^2 \mathrm{d}x + \fint_{\Omega_\lambda} |R_\varepsilon^\mu-\nabla y(x)|^2 \mathrm{d}x \right) \\&\leq C_{\lambda,\mu}\varepsilon \left( E_\varepsilon(y_\varepsilon,u_\varepsilon,h_\varepsilon)+ |\Omega_h| \right)\leq C_{\lambda,\mu}\varepsilon
\end{split}
\end{align}
and again by convexity and (\ref{Poincare}) we have
\begin{align}\label{translationdifference}
\nonumber
|b_\varepsilon^\lambda-b_\varepsilon^\mu|^2 &\leq C_{\lambda,\mu}\Big(|R_\varepsilon^\lambda-R_\varepsilon^\mu|^2 +\fint_{\Omega_\lambda} |R_\varepsilon^\lambda x +b_\varepsilon^\lambda-y(x)|^2 \mathrm{d}x   +\fint_{\Omega_\lambda} |R_\varepsilon^\mu x +b_\varepsilon^\mu-y(x)|^2 \mathrm{d}x   \Big)\\&\leq C_{\lambda,\mu}\left(\varepsilon + \fint_{\Omega_\lambda} |R_\varepsilon^\lambda-\nabla y(x)|^2 \mathrm{d}x + \fint_{\Omega_\lambda} |R_\varepsilon^\mu-\nabla y(x)|^2 \mathrm{d}x\right) \leq C_{\lambda,\mu}\varepsilon.
\end{align}
Assume that $\Omega_1 \neq \emptyset$ and define $u_\varepsilon : \Omega_h \to \mathbb{R}^2 $ by
\begin{align*}
u_\varepsilon(x) = \frac{y_\varepsilon(x)-(R_\varepsilon^1 x +b_\varepsilon^1)}{\sqrt{\varepsilon}}.
\end{align*}
Now again by convexity (\ref{Rigidity}),(\ref{Rotationdifference}) and (\ref{translationdifference}) we obtain for $\lambda>0$
\begin{align}\label{ineq : uepsbound}
\begin{split}
||u_\varepsilon||_{H^1(\Omega_\lambda)}^2 &\leq  \frac{C_\lambda}{\varepsilon}\left(\int_{\Omega_\lambda}|\nabla y_\varepsilon(x) -R_\varepsilon^\lambda|^2\mathrm{d}x +|b_\varepsilon^1-b_\varepsilon^\lambda|^2 +|R_\varepsilon^1-R_\varepsilon^\lambda|^2 \right) \\&\leq C_{\lambda,1} \left( E_\varepsilon(y_\varepsilon,u_\varepsilon,h_\varepsilon)+1 \right) \leq C_\lambda
\end{split}
\end{align}
and the claim follows.  This implies (\ref{Compactnessu}) and by the Rellich--Kondrachov Theorem ii) as well as iii). It remains to prove that $u(0,x_2) = u(L,x_2) $ for $\mathcal{H}^1$-a.e. $x_2 \in (-R,h(0))$. Now by the definition of $u_\varepsilon$ we have that $u_\varepsilon(0,x_2)= u_\varepsilon(L,x_2)$ for $\mathcal{H}^1$-a.e. $x_2 \in (-R,h_\lambda(0))$. Now by (\ref{ineq : uepsbound}) we have $||u_\varepsilon||_{H^1(\Omega_\lambda)} \leq C_\lambda $ and therefore also $u_\varepsilon \rightharpoonup u$ weakly in $H^1(\Omega_\lambda)$. By the continuity of the trace operator with respect to weak convergence in $H^1$ we have that $ u(0,x_2)= u(L,x_2)$ for $\mathcal{H}^1$-a.e. $x_2 \in (-R,h_\lambda(0))$. Now, since $h_\lambda \to h$ pointwise, for any $ x_2 < h(0)$ there exists $\lambda>0$ such that $x_2 \leq  h_\lambda(0)$. We can therefore conclude that $ u(0,x_2)= u(L,x_2)$ for $\mathcal{H}^1$-a.e. $x_2 \in (-R,h(0))$. This concludes the  proof. 

\end{proof}

\section{Asymptotic Analysis}

In this section we state the main result and perform all the analysis to proof it in a rigorous way.

\begin{theorem}[Main Theorem] \label{main_theorem} Let $\varepsilon \to 0$ and let $\delta_\varepsilon \to 0$ satisfy (\ref{Mismatch}). Let $E_\varepsilon : X \to [0,+\infty]$ be defined by (\ref{Definition Energy2}) and $E : X \to [0,+\infty]$ be defined by (\ref{Definition Limit Energy}). Then $E_\varepsilon$  $\Gamma$-converges to $E$ with respect to the topology defined in Definition\ref{Definition Convergence}.
\end{theorem}
\begin{proof} The proof follows  from the definition of $\Gamma$-convergence (see Section \ref{setting}),  Proposition \ref{Proposition Gamma-liminf} and Proposition \ref{Proposition Gamma-limsup}.

\end{proof}

\subsection{Lower Bound}

In this chapter we proof the $\Gamma$-$\liminf$-inequality,  i.e., for all $\{x_\varepsilon\}_\varepsilon \subset X$ converging to $x\in X$ with respect to $d$ there holds
\begin{align*}
\liminf_{\varepsilon \to 0} E_\varepsilon(x_\varepsilon) \geq E(x).
\end{align*}

 The proof of the liminf-inequality decouples into proving both the liminf-inequality for the surface part as well as the elastic part of the energy. 
 
 \begin{proposition}[$\Gamma$-$\liminf$-inequality]\label{Proposition Gamma-liminf} Let $(y,u,h)\in X$. We have 
\begin{align*}
E'(y,u,h) \geq E(y,u,h).
\end{align*}
\end{proposition}
\begin{proof}
The Proof follows from Proposition \ref{Proposition Gamma-liminf Surface} and Proposition \ref{Proposition Gamma-liminf Elastic} and  by  applying the super-additivity of the $\Gamma$-liminf (see  \cite[Proposition 6.17]{dal2012introduction}). We have
\begin{align*}
&\Gamma\text{-}\liminf_{\varepsilon \to 0 } E^{el}_\varepsilon(y,u,h) \geq  E^{el}(y,u,h) \\&\Gamma\text{-} \liminf_{\varepsilon \to 0 } E^{S}_\varepsilon(h) \geq  E^{S}(h).
\end{align*}
Now since $E_\varepsilon(y,u,h) = E_\varepsilon^{el}(y,u,h) +E_\varepsilon^S(y,u,h)$ we have
\begin{align*}
E'(y,u,h) \geq \Gamma\text{-}\liminf_{\varepsilon \to 0 } E^{el}_\varepsilon(y,u,h) +\Gamma\text{-} \liminf_{\varepsilon \to 0 } E^{S}_\varepsilon(h)  \geq E^{el}(y,u,h) + E^S(h) =E(y,u,h)
\end{align*}
and the claim follows.
\end{proof}

\medskip

 For the surface part we need a semi-continuity result for a class of functionals $F$  defined on the family of sets 
\begin{equation}\label{setfunctional}
\mathcal{A}_c:=\{A\subset\bar{Q} :\,\,\text{$\partial A$ is 
$\mathcal{H}^1$-rectifiable, connected, and $\mathcal{H}^1(\partial A)<+\infty$}\}.
\end{equation}
More precisely, we consider $F: \mathcal{A}_c(\mathbb{R}^2) \to [0,+\infty]$   defined by 
 \begin{align*}
F(\Omega) := \int_{\partial \Omega \cap \Omega^{1/2}}\psi(x,\nu)\mathrm{d}\mathcal{H}^1 + 2\int_{\partial \Omega \cap \Omega^{1}}\psi(x,\nu)\mathrm{d}\mathcal{H}^1
\end{align*}
for every $\Omega\in \mathcal{A}_c$, where $\psi: \mathbb{R}^2\times \mathbb{R}^2\to  [0,+\infty)$ is a continuous surface tensions. Such a result, which we recall here for reader's convenience, is obtained in  \cite{KP} as a corollary from a more general setting which includes not only thin films, but also other stress-driven rearrangement instabilities. 
Notice that in \eqref{setfunctional}  $\nu = \tau^\perp$, which exists and it is well-defined for $\mathcal{H}^1$-a.e. $x \in \partial \Omega$ whenever $\partial \Omega$ is connected and $\mathcal{H}^1(\partial \Omega) <+\infty$. 

\medskip

\begin{theorem}\label{SurfaceSemicontinuity} Let  $R'>0$ and $Q_{R'}:=Q\cap \{x_2<R'\}$. If $\varphi\in\mathcal{C}(\overline{Q_{R'}} \times \mathbb{R}^2; [0,+\infty))$  is convex, even, positively $1$-homogeneous in the second argument, and there exists $C>0$ such that
\begin{align*}
\frac{1}{C}|\nu|\leq \varphi(x,\nu)\leq C|\nu|,
\end{align*} 
for every $\nu \in \mathbb{R}^2$, then  $F$ is lower-semicontinuous with respect to the Hausdorff convergence  of the complements of sets, i.e.,\begin{align*}
\liminf_{n \to +\infty} F(\Omega_n)\geq F(\Omega)
\end{align*}
whenever $\overline{Q_{R'}}\setminus \Omega_n \to \overline{Q_{R'}}\setminus \Omega$ with respect to the Hausdorff-distance.
\end{theorem}

We are now ready to prove the $\Gamma$-$\liminf$-inequality for the surface energy.

\begin{proposition}[$\Gamma$-$\liminf$-surface-inequality] We have \label{Proposition Gamma-liminf Surface}
\begin{align*}
\Gamma\text{-}\liminf_{\varepsilon \to 0} E^S_\varepsilon(h) \geq E^S(h).
\end{align*}
\end{proposition}
\begin{proof} Note that if $h_\varepsilon \to h$ we have that $Q\setminus \Omega_{h_\varepsilon} \to Q\setminus \Omega_h$ with respect to the Hausdorff-convergence of sets, $||h||_{L^1(S^1_L)}=V$ and 
\begin{equation}\label{semicontinuity1}
\sup_{\varepsilon >0 } \mathrm{Var}h_\varepsilon <+\infty.
\end{equation}
Furthermore we can assume that
\begin{equation}\label{semicontinuity2}
\sup_{\varepsilon >0} E^S_\varepsilon(h_\varepsilon) < +\infty.
\end{equation}
For $i \in \mathcal{L}_\varepsilon$ denote by $\mathcal{V}_\varepsilon(i)$ the \textit{Voronoi cell} of $i$ in $\mathcal{L}_\varepsilon$ given by
\begin{align*}
\mathcal{V}_\varepsilon(i)=\{x \in \mathbb{R}^2 : |x-i|\leq |x-j| \text{ for all } j \in \mathcal{L}_\varepsilon\}.
\end{align*}
Define $\Omega_\varepsilon \subset \mathbb{R}^2$ by
\begin{align*}
\Omega_\varepsilon = \bigcup_{i \in \mathcal{L}_\varepsilon(\Omega_{h_\varepsilon})} \mathcal{V}_\varepsilon(i).
\end{align*} 

Notice that from \eqref{semicontinuity1} and  \eqref{semicontinuity2} we deduce that $h_\varepsilon$ are uniformly bounded in $BV(0,L)$ and hence,  there exists  $R'>0$ such that  $\Omega_\varepsilon\subset Q_{R'}$, $\Omega\subset Q_{R'}$, and  $\overline{Q_{R'}}\setminus \Omega_{h_\varepsilon} \to \overline{Q_{R'}}\setminus \Omega_h$, where $Q_{R'}:=Q\cap \{x_2<R'\}$.
We also observe that    there holds
\begin{align*}
\mathrm{dist}_{\mathcal{H}}(\mathcal{L}_\varepsilon(\Omega_{h_\varepsilon}),\Omega_{h_\varepsilon})\leq C\varepsilon,\quad \mathrm{dist}_{\mathcal{H}}(\mathcal{L}_\varepsilon(\Omega_{h_\varepsilon}),\Omega_\varepsilon)\leq C\varepsilon
\end{align*}
and therefore $Q\setminus\Omega_\varepsilon \to Q\setminus\Omega_h$ with respect to the Hausdorff-distance. Now fix $\eta>0$ and define $\varphi_\eta : \mathbb{R}^2\times \mathbb{R}^2 \to [0,+\infty] $ by
\begin{align*}
\varphi_\eta((x_1,x_2),\nu) = \begin{cases} \gamma_f \varphi(\nu)& x_2 >2\eta\\
(t\gamma_f + (1-t)\gamma_f \wedge \gamma_s) \varphi(\nu) &x_2= t 2\eta + (1-t)\eta, t \in (0,1)\\
 (\gamma_f \wedge \gamma_s) \varphi(\nu) &\text{otherwise,}
\end{cases}
\end{align*}
with $\varphi$ defined by (\ref{Definitionphi}). There holds
\begin{align*}
E^S_\varepsilon(h_\varepsilon) \geq \int_{\partial\Omega_{\varepsilon}\cap\Omega_{\varepsilon}^{1/2}}\varphi_\eta(x,\nu)\mathrm{d}\mathcal{H}^1 + 2\int_{\partial\Omega_{\varepsilon}\cap \Omega_{\varepsilon}^{1}}\varphi_\eta(x,\nu)\mathrm{d}\mathcal{H}^1 =: E_\eta(\Omega_\varepsilon).
\end{align*}
By Theorem \ref{SurfaceSemicontinuity} we have that $E_\eta$ is lower-semicontinuous and therefore 
\begin{align*}
\liminf_{\varepsilon\to 0} E^S_\varepsilon(h_\varepsilon)  \geq \liminf_{\varepsilon\to 0}E_\eta(\Omega_\varepsilon) \geq E_\eta(\Omega_h).
\end{align*}
Using the monotone convergence theorem we obtain that $E_\eta \to E^S$ increasingly as $\eta \to 0$ and therefore we have
\begin{align*}
\liminf_{\varepsilon\to 0} E^S_\varepsilon(h_\varepsilon)  \geq \sup_{\eta >0} E_\eta(\Omega_h) =E^S(h).
\end{align*}
Hence the claim follows.

\end{proof}

 For the elastic part of the energy we localize first on sets $\Omega'\subset\subset \Omega_h$ in order to have good convergence properties of $u_\varepsilon $ to $u$ (namely weakly in $H^1(\Omega')$ by the previous compactness proposition). We then localize on sets where the gradient of the rescaled displacement is less than a certain threshold $k_\varepsilon$.  The threshold   $k_\varepsilon$ is suitably  chosen so that one can Taylor expand on the set where the gradient is less than $k_\varepsilon$, 
and show that the set invades the whole $\Omega'$. 

\begin{proposition}[$\Gamma$-$\liminf$-elastic-inequality] We have \label{Proposition Gamma-liminf Elastic}
\begin{align*}
\Gamma\text{-}\liminf_{\varepsilon \to 0} E^{el}_\varepsilon(y,u,h) \geq E^{el}(y,u,h).
\end{align*}
\end{proposition}
\begin{proof} Let $\{y_\varepsilon,u_\varepsilon, h_\varepsilon\} \subset X$ converge to $(y,u,h)$. Without loss of generality we can assume that 
\begin{align*}
\sup_{\varepsilon >0} E_{\varepsilon}(y_\varepsilon, u_\varepsilon,h_\varepsilon) <+\infty.
\end{align*} 
Furthermore we assume that
\begin{align*}
\lim_{\varepsilon\to 0} E^{el}_{\varepsilon}(y_\varepsilon,u_\varepsilon, h_\varepsilon) = \liminf_{\varepsilon\to 0}E^{el}_{\varepsilon}(y_\varepsilon,u_\varepsilon,h_\varepsilon).
\end{align*}
By Proposition \ref{Compactness Proposition} we have that there exist $\{R_\varepsilon\}_\varepsilon \subset SO(2),\{ b_\varepsilon\}_\varepsilon \subset \mathbb{R}^2$ such that the functions $u_\varepsilon : \mathcal{L}_\varepsilon(\Omega_h)\to \mathbb{R}^2 $ defined by
\begin{align*}
u_\varepsilon(x) = \frac{y_\varepsilon(x) -(R_\varepsilon x+b_\varepsilon)}{\sqrt{\varepsilon}} 
\end{align*}
converge to $u$ in $H^1_{\mathrm{loc}}(\Omega_h)$. Fix $\Omega'\subset\subset \Omega_h$. For $\varepsilon>0$ small enough there holds $\Omega' \subset\subset \Omega_{h_\varepsilon}$ we therefore have 
\begin{align*}
E_{\varepsilon}^{el}(y_\varepsilon,u_\varepsilon,h_\varepsilon) &\geq  \underset{T \cap\Omega'\cap \Omega_h^+ \neq \emptyset}{\sum_{T \in \mathcal{T}_\varepsilon}}\varepsilon W_{\varepsilon,cell}(\nabla y_\varepsilon(\hat{x}_T),T) + \underset{T \cap\Omega'\cap \Omega^- \neq \emptyset}{\sum_{T \in \mathcal{T}_\varepsilon}}\varepsilon W_{\varepsilon,cell}(\nabla y_\varepsilon(\hat{x}_T),T)\\&= I_{\varepsilon}^+ + I_\varepsilon^-.
\end{align*}
Now fix $\eta >0$ and define $\Omega^+_\eta= \left(\Omega' \cap\Omega_{h}^+\cap\{x_2 > \eta\} \right) \subset \subset \Omega_{h_\varepsilon}^+$. For every $T \in \mathcal{T}_\varepsilon$ such that $T \cap \Omega^+_\eta \neq \emptyset$ we have that $W_{\varepsilon,cell}(F,T) = \frac{K_f}{4} W_{\lambda_\varepsilon}(F)$. Now
\begin{align}\label{Iepsminusfirstestimate}
\begin{split}
I_\varepsilon^+ &\geq \varepsilon^{-1} \frac{4}{\sqrt{3}}\underset{T \cap \Omega^+_\eta \neq  \emptyset} {\sum_{T \in \mathcal{T}_\varepsilon}}\int_{T}W_{\varepsilon,cell}(\nabla y_\varepsilon(x),T)\mathrm{d}x = \varepsilon^{-1} \frac{16}{\sqrt{3}}K_f\underset{T \cap \Omega^+_\eta \neq \emptyset}{\sum_{T \in \mathcal{T}_\varepsilon}}\int_{T}W_{1+\delta_\varepsilon}(\nabla y_\varepsilon(x))\mathrm{d}x \\&\geq \varepsilon^{-1}\frac{16}{\sqrt{3}} K_f\int_{\Omega^+_\eta}W_{1+\delta_\varepsilon}(\nabla y_\varepsilon(x))\mathrm{d}x\\&= \varepsilon^{-1}\frac{16}{\sqrt{3}} K_f\int_{\Omega^+_\eta}W_{1+\delta_\varepsilon}\left((1+\delta_\varepsilon)R_\varepsilon + \sqrt{\varepsilon}(\nabla u_\varepsilon - \frac{\delta_\varepsilon}{\sqrt{\varepsilon}}R_\varepsilon)  \right)\mathrm{d}x \\&\geq \varepsilon^{-1}\frac{16}{\sqrt{3}}  K_f\int_{\Omega^+_\eta}\chi_\varepsilon(x) W_{1+\delta_\varepsilon}\left((1+\delta_\varepsilon)R_\varepsilon + \sqrt{\varepsilon}(\nabla u_\varepsilon - \frac{\delta_\varepsilon}{\sqrt{\varepsilon}}R_\varepsilon)  \right)\mathrm{d}x
\end{split}
\end{align}
where we set $\chi_{\varepsilon}(x) = \chi_{\{|\nabla u_\varepsilon|(x) \leq k_\varepsilon\}}(x)$ with $k_\varepsilon >0$ to be chosen later. Now by Taylor expanding $W_{1+\delta_\varepsilon}$ around $(1+\delta_\varepsilon)R_\varepsilon$, using the assumptions on $W_{1+\delta_\varepsilon}$ one can check that $D^2W(F)=  D^2 W_{1+\delta_\varepsilon}((1+\delta_\varepsilon)R_\varepsilon)(F,F)= D^2 W_1(R_\varepsilon)(F,F)$. Therefore we obtain we obtain
\begin{align*}
W_{1+\delta_\varepsilon}\left((1+\delta_\varepsilon)R_\varepsilon + F \right) = \frac{1}{2} D^2 W(F) + w\left(|F|\right),
\end{align*}
where $\sup \{\frac{w(F)}{|F|^2}: |F| \leq \rho\}\to 0$ as $\rho \to 0$ independent of $\varepsilon$. Using (\ref{Iepsminusfirstestimate}) we have that
\begin{align*}
I_\varepsilon^+ \geq\frac{8}{\sqrt{3}}K_f\int_{\Omega_\eta^+}D^2W\left(\chi_\varepsilon(x)(\nabla u_\varepsilon - \frac{\delta_\varepsilon}{\sqrt{\varepsilon}}R_\varepsilon)\right)+\varepsilon^{-1}\chi_\varepsilon(x)w\left(\sqrt{\varepsilon}|\nabla u_\varepsilon-\frac{\delta_\varepsilon}{\sqrt{\varepsilon}}R_\varepsilon |\right) \mathrm{d}x.
\end{align*}
The second term is bounded by
\begin{align*}
|\nabla u_\varepsilon-\frac{\delta_\varepsilon}{\sqrt{\varepsilon}}R_\varepsilon |^2\chi_\varepsilon(x)\frac{ w\left(\sqrt{\varepsilon}|\nabla u_\varepsilon-\frac{\delta_\varepsilon}{\sqrt{\varepsilon}}R_\varepsilon |\right)}{\varepsilon |\nabla u_\varepsilon-\frac{\delta_\varepsilon}{\sqrt{\varepsilon}}R_\varepsilon |^2}.
\end{align*}
If we choose $k_\varepsilon \to +\infty$ such that $k_\varepsilon \sqrt{\varepsilon} \to 0$, then $|\nabla u_\varepsilon \frac{\delta_\varepsilon}{\sqrt{\varepsilon}}R_\varepsilon|$ is bounded in $L^2(\Omega')$ and $\chi_\varepsilon(x)\frac{ w\left(\sqrt{\varepsilon}|\nabla u_\varepsilon-\frac{\delta_\varepsilon}{\sqrt{\varepsilon}}R_\varepsilon |\right)}{\varepsilon |\nabla u_\varepsilon-\frac{\delta_\varepsilon}{\sqrt{\varepsilon}}R_\varepsilon |^2}$ converges to zero uniformly in $\varepsilon$. We therefore deduce that
\begin{align*}
\liminf_{\varepsilon \to 0 }I_\varepsilon^+ \geq \liminf_{\varepsilon \to  0} \frac{8}{\sqrt{3}}K_f\int_{\Omega_\eta^+}D^2W\left(\chi_\varepsilon(x)(\nabla u_\varepsilon - \frac{\delta_\varepsilon}{\sqrt{\varepsilon}}R_\varepsilon)\right)\mathrm{d}x
\end{align*}
Noting that $\frac{\delta_\varepsilon}{\sqrt{\varepsilon}} \to \delta, R_\varepsilon \to  R$ and $\chi_\varepsilon$ converges to $1$ in measure in $\Omega_\eta^+$ we have
\begin{align*}
\chi_{\Omega_\eta^+} \chi_\varepsilon\left( \nabla u_\varepsilon - \frac{\delta_\varepsilon}{\sqrt{\varepsilon}}R_\varepsilon \right) \rightharpoonup \nabla u - \delta  R \text{ in } L^2(\Omega_\eta^+)
\end{align*}
By lower-semicontiuity of convex functionals with respect to weak convergence we obtain
\begin{align*}
\liminf_{\varepsilon \to 0} I_\varepsilon^+ \geq \frac{8}{\sqrt{3}}K_f \int_{ \Omega_\eta^+} D^2W\left(\nabla u -\delta  R\right)\mathrm{d}x.
\end{align*}
Now letting $\eta \to 0$ we obtain
\begin{align}\label{liminfelfilm}
\liminf_{\varepsilon \to 0} I_\varepsilon^+ \geq \frac{8}{\sqrt{3}}K_f \int_{ \Omega_h^+ \cap \Omega'} D^2W\left(\nabla u -\delta  R\right)\mathrm{d}x.
\end{align}
The proof of
\begin{align}\label{liminfelsubstrate}
\liminf_{\varepsilon\to 0} I_\varepsilon^- \geq \frac{8}{\sqrt{3}}K_s \int_{\Omega'\cap \Omega^-} D^2W\left(\nabla u\right)\mathrm{d}x
\end{align}
follows exactly the same steps. Using (\ref{liminfelfilm}) and (\ref{liminfelsubstrate}) we obtain
\begin{align*}
\liminf_{\varepsilon \to 0} E_{\varepsilon}^{el}(y_\varepsilon,u_\varepsilon,h_\varepsilon) &\geq \liminf_{\varepsilon \to 0}I_\varepsilon^+ + \liminf_{\varepsilon \to 0} I_\varepsilon^- \\&\geq \frac{8}{\sqrt{3}}\left(K_f \int_{\Omega'\cap \Omega_h^+} D^2W\left(\nabla u -\delta  R\right)\mathrm{d}x +K_s \int_{\Omega'\cap \Omega^-} D^2W\left(\nabla u\right)\mathrm{d}x \right)\\&=\int_{\Omega'\cap \Omega_h} W_y(x,Eu)\mathrm{d}x
\end{align*}
Letting $\Omega' \to \Omega_h$ we obtain the claim.
\end{proof}

\subsection{Upper Bound}

In this section we prove the $\Gamma$-$\limsup$-inequality,  i.e., for all $x \in X$ there exists a sequence  $\{x_\varepsilon\}_\varepsilon $ converging to $x$ with respect to $d$ such that
\begin{align*}
\limsup_{\varepsilon \to 0} E_\varepsilon(x_\varepsilon)\leq E(x). 
\end{align*}

 The result is based on density results.  
 In a first step we use the Yosida-transform of the profile $h$ in order to obtain a Lipschitz approximation. 
\medskip

 To this aim we begin by observing that $\varphi : \mathbb{R}^2 \to [0,+\infty)$  defined in \ref{Definitionphi} is a convex, even,  and positively homogeneous function of degree one and such that
\begin{align}\label{phinorm}
\frac{1}{C}|\nu|\leq \varphi(\nu)\leq C|\nu|.
\end{align}

Furthermore, let us define $h^- : S^1_L \to [0,+\infty)$  by 
\begin{align*}
h^-(x) := \inf\{\liminf_{n \to +\infty} h(x_n) : x_n \to x, x_n \neq x\}.
\end{align*}
If $\mathrm{Var} h < +\infty$, then $\{x : h(x) < h^-(x)\}$ is at most countable.

\begin{lemma}\label{Prop SurfaceContinuousApprox}  Let $h : S^1_L \to \mathbb{R}_+$ be a lower semicontinuous function such that  $||h||_{L^1(S^1_L)} <+\infty$.  
Then there exist a sequence of Lipschitz functions $h_n :S^1_L \to \mathbb{R}_+$ such that $h_n\leq h_{n+1}\leq h$,  $h_n\to h$ in $L^1(S^1_L)$,  $Q \setminus \Omega_{h_n} \to Q \setminus \Omega_h$ with respect to the Hausdorff-distance and 
\begin{align*}
E^S(h_n) \to E^S(h),
\end{align*}
 where $E^S$ is defined by (\ref{ESenergy}). 
\end{lemma}
\begin{proof} Define $h_n : S^1_L \to \mathbb{R}_+$ as the Yosida-transform of $h$ given by
\begin{align*}
h_n(x) = \inf \left\{h(y) + n|x-y | \right\}.
\end{align*}
We then have that $h_n\leq h_{n+1}\leq h$ and $h_n \to h$ pointwise  and hence, $h_n\to h$ in $L^1(S^1_L)$ and in the sense of $\Gamma$-convergence  (cf. \cite{dal2012introduction}).  Let us assume without loss of generality that $E^S(h) < +\infty$. By (\ref{phinorm}) we have that $\mathrm{Var} h  <+\infty$, which in turns together with  $||h||_{L^1(S^1_L)} <+\infty$  yields that  $||h||_\infty < +\infty$.  One can check that $Q\setminus \Omega_{h_n} \to Q\setminus \Omega_h$ with respect to the Hausdorff-distance.
 Now fix $\eta>0$ and define $\varphi_\eta : \mathbb{R}^2\times \mathbb{R}^2 \to [0,+\infty] $ by
\begin{align*}
\varphi_\eta((x_1,x_2),\nu) = \begin{cases} \gamma_f \varphi(\nu)& x_2 >2\eta\\
(t\gamma_f + (1-t)\gamma_f \wedge \gamma_s) \varphi(\nu) &x_2= t 2\eta + (1-t)\eta, t \in (0,1)\\
 (\gamma_f \wedge \gamma_s) \varphi(\nu) &\text{otherwise,}
\end{cases}
\end{align*}
with $\varphi$ defined by (\ref{Definitionphi}). Now define $E^S_\eta : AP([0,L]) \to [0,+\infty]$ by
\begin{align*}
E^S_\eta(h) = \int_{\partial\Omega_h\cap\Omega^{1/2}_h}\varphi_\eta(x,\nu)\mathrm{d}\mathcal{H}^1 + 2\int_{\partial\Omega_h\cap \Omega^{1}_h}\varphi_\eta(x,\nu)\mathrm{d}\mathcal{H}^1.
\end{align*}

 Since $Q \setminus \Omega_{h_n} \to Q \setminus \Omega_h$ with respect to the Hausdorff-distance and Theorem \ref{SurfaceSemicontinuity} we have that
\begin{align*}
\liminf_{n \to +\infty} E^S_\eta(h_n) \geq E^S_\eta(h).
\end{align*}
Now by the Monotone Convergence Theorem there holds 
\begin{align*}
\lim_{\eta \to 0} E^S_\eta(h) = \sup_{\eta >0} E^S_\eta(h) =  E^S(h)
\end{align*}
for all $h \in AP(S^1_L)$. Now
\begin{align*}
\liminf_{n \to +\infty} E^S(h_n) \geq   \liminf_{n \to +\infty} E^S_\eta(h_n) \geq  E^S_\eta(h).
\end{align*}
Letting $\eta \to 0$ we therefore have
\begin{align*}
\liminf_{n \to +\infty} E^S(h_n) \geq  E^S(h). 
\end{align*}

It remains to prove the other inequality. We first prove that $\{h_n =0 \} = \{h=0\}$. Suppose that $h_n(x) =0$, then there exists $x'\in [0,L]$ such that $0=h(x') + n|x-x'| \geq n|x-x'| \geq 0$. Hence $x'=x$ and $h(x)=0$. On the other hand if $h(x)=0$, since $0\leq h_n(x) \leq h(x)=0$ it follows $h_n(x)=0$. And therefore the claim is proven. Moreover this set is essentially $\left(\partial \Omega_h \cap \Omega_h^{\frac{1}{2}}\right) \setminus Q^+$.
Noting that the set $\Omega_h$  is Lipschitz we have that $\partial \Omega_h =\partial \overline{\Omega}_h $ and using $\{h_n =0 \} = \{h=0\}$ we obtain
\begin{align}\label{EnergyEquality}
\begin{split}
E^S(h_n) &= \gamma_f\int_{\partial \Omega_{h_n}\cap Q^+} \varphi(\nu)\mathrm{d}\mathcal{H}^1+  \gamma_f\wedge  \gamma_s\int_{\partial\overline{\Omega}_{h_n}\setminus Q^+} \varphi(\nu)\mathrm{d}\mathcal{H}^1 \\&=  \gamma_f\int_{\partial \Omega_{h_n}\cap Q^+} \varphi(\nu)\mathrm{d}\mathcal{H}^1+  \gamma_f\wedge  \gamma_s\int_{\partial\overline{\Omega}_h\setminus Q^+} \varphi(\nu)\mathrm{d}\mathcal{H}^1.
\end{split}
\end{align}
We now need to estimate the first term. To this end we split $\partial \Omega_{h_n} $ into two parts, $\partial \Omega_{h_n} \cap \partial \Omega_h$ and $\partial \Omega_{h_n} \setminus \partial \Omega_h$. First notice that $\partial \Omega_h \cap \partial \Omega_{h_n}$ is essentially equal to $\partial \overline{\Omega}_h  \cap \partial \Omega_{h_n}$. Indeed let $(x_1,x_2) \in \partial \Omega_{h_n} \cap (\partial\Omega_h \setminus \partial\overline{\Omega}_h)$. We have $x_2 = h_n(x_1) \leq h(x_1)$ and $h(x_1) \leq x_2 < h^-(x_1)$, thus $x_2 = h(x_1)$ and $h(x_1) < h^-(x_1)$, but this happens for at most a countable number of points. Thus $\partial \Omega_{h_n} \cap (\partial\Omega_h \setminus \partial\overline{\Omega}_h) $ is at most countable. Therefore we can write
\begin{align}\label{Energyequality2}
 \gamma_f\int_{\partial \Omega_{h_n}\cap Q^+} \varphi(\nu)\mathrm{d}\mathcal{H}^1 =  \gamma_f\int_{\partial \overline{\Omega}_{h}\cap \Omega_{h_n}\cap Q^+} \varphi(\nu)\mathrm{d}\mathcal{H}^1 +  \gamma_f\int_{\partial \Omega_{h_n} \setminus \partial\Omega_h\cap Q^+} \varphi(\nu)\mathrm{d}\mathcal{H}^1.
\end{align}
It remains to estimate the second term on the right hand side. We have $L_n=\{x_1: (x_1,x_2) \in \partial \Omega_{h_n} \setminus \partial \Omega_h \} = \{x_1 \in [0,L] :  h_n(x_1) < h(x_1)\} $ which is an open set. It can be written as a disjoint union of open intervals:
\begin{align*}
L_n = \bigcup_{k=1}^K(a_k,b_k),
\end{align*}
with $a_k=a_k^n,b_k=b_k^n \in [0,L]$ $ a_k < b_k \leq a_{k+1}$ for all $k \in \{1,\ldots,K\}$ and $K=K(n)\in \mathbb{N}_0\cup \{+\infty\}$. Fix $k \in \mathbb{N}$ and consider $(a_k,b_k)$. We claim that for $x \in (a_k,b_k)$ there holds
\begin{align}\label{hnrepresentation}
h_n(x) = \min\{h(a_k) +n(x-a_k),h(b_k)+n(b_k-x)\}.
\end{align}
Let us prove (\ref{hnrepresentation}). By testing with $a_k,b_k $ respectively we obtain that 
\begin{align*}
h_n(x) \leq \min\{h(a_k) +n(x-a_k),h(b_k)+n(b_k-x)\}.
\end{align*}
Now for the other inequality. Assume there exists $x' < a_k$ such that $h_n(x) = h(x') + n|x-x'| < h(a) + n|x-a|$ but this contradicts $h_n(a) =h(a)$. The same way we can prove that there does not exist $x' >b_k$ such that $h_n(x) = h(x')+n|x-x'|$. It remains to prove the claim for $x \in (a_k,b_k)$. Let us assume that there exist $x' \in (a_k,b_k) $ such that $h_n(x) = h(x') + n|x-x'|$. Then, since $ \mathrm{Lip} h_n \leq n$ we have
\begin{align*}
 h(x') + n|x-x'|= h_n(x) \leq h_n(x') + n|x-x'| \leq h(x') + n|x-x'|
\end{align*}
so that $h_n(x')=h(x')$ which contradicts $x' \in (a_k,b_k)$. There are now two cases to consider:
\begin{itemize}
\item[a)] $h_n(x) = h(a_k) + n|x-a_k| $ (or similar $h_n(x) = h(b_k) + n|x-b_k| $)
\item[b)] There exists $x_0=x_0^k \in (a_k,b_k)$ such that
\begin{align*}
h_n(x) = \begin{cases} h(a_k) + n|a_k-x| &x \in [a_k,x_0),\\
 h(b_k) + n|b_k-x| &x \in [x_0,b_k).
\end{cases}
\end{align*}
\end{itemize}
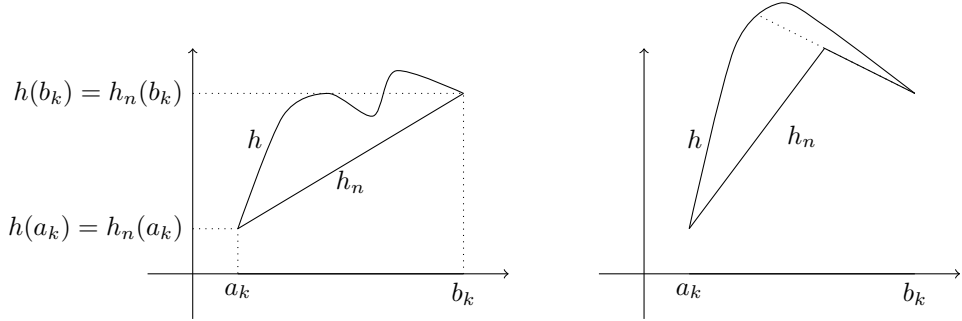
\begin{figure}[htp]
\centering
\begin{tikzpicture}[scale=0.6]
\draw[->](-2,0)--++(8,0);
\draw[->](-1,-1)--++(0,6);
\draw(0,0)--++(5,0) node[anchor =north]{$b_k$};
\draw(0,0) node[anchor =north]{$a_k$};
\draw(0,1)--++(5,3);
\draw(0,1)++(2.5,1.5) node[anchor = north]{$h_n$};
\draw[dotted](-1,1) node[anchor=east]{$h(a_k)=h_n(a_k)$}--++(1,0)--++(0,-1);
\draw[dotted](-1,4) node[anchor=east]{$h(b_k)=h_n(b_k)$}--++(6,0)--++(0,-4);
\draw  plot [smooth] coordinates {(0,1)(1,3.5)(2,4)(3,3.5)(3.5,4.5)(5,4) };
\draw(0.75,3) node[anchor =east]{$h$};

\begin{scope}[shift={(10,0)}]
\draw[->](-2,0)--++(8,0);
\draw[->](-1,-1)--++(0,6);
\draw(0,0)--++(5,0) node[anchor =north]{$b_k$};
\draw(0,0) node[anchor =north]{$a_k$};
\draw(0,1)--(3,5)--(5,4);
\draw[dotted](5,4)--++(-3.5,1.75);
\draw(0.5,3) node[anchor =east]{$h$};
\draw(0,1)++(2.5,2.5) node[anchor = north]{$h_n$};
\draw  plot [smooth] coordinates {(0,1)(1,5)(2,6)(3,5.5)(5,4) };
\end{scope}

\end{tikzpicture}
\caption{The case a) on the left and the case b) on the right}
\end{figure}
In the first case the graph of $h_n$ in $[a_k,b_k] \times \mathbb{R}$ is a straight line $l_k$ from $(a_k,h(a_k))$ to $(b_k,h(b_k))$, while the  set  $\partial \Omega_h \cap ([a_k,b_k] \times \mathbb{R})$ contains a curve $\gamma_k : (0,L_k) \to [a_k,b_k]\times \mathbb{R}$ parametrized by arc-length connecting those two points, that is $ \gamma_k(0) = (a_k,h(a_k)), \gamma_k(L_k) = (b_k,h(b_k)) $ and by \cite{falconer1986geometry}, Lemma 3.12 and Theorem 3.8 the curve has a tangent $\dot{\gamma}_k$ at $\mathcal{H}^1$-a.e. point and it satifies $|\dot{\gamma}_k|=1$ $\mathcal{L}^1$-a.e. in $(0,L_k)$. We set $\xi_k =(a_k,h(a_k))-(b_k,h(b_k))$. Now by convexity and the positively one homogeneity we have
\begin{align}\label{energyinequality1}
\begin{split}
\int_{\partial \Omega_{h_n} \cap [a_k,b_k]\times\mathbb{R}}\varphi(\nu)\mathrm{d}\mathcal{H}^1&=\int_{l_k} \varphi(\nu_{l_k})\mathrm{d}\mathcal{H}^1= |\xi_k| \varphi\left(\frac{\xi_k^\perp}{|\xi_k|}\right)=\varphi(\xi_k^\perp) \\&= \varphi\Big(\left(\int_0^{L_k}\dot{\gamma}_k(t)\mathrm{d}t\right)^\perp\Big) \leq \int_0^{L_k}\varphi(\dot{\gamma}_k^\perp(t))\mathrm{d}t.
\end{split}
\end{align}
In case b) the graph of $h_n$ is made of two straight lines, one going from $(a_k,h(a_k))$ to $(x_0,h_n(x_0))$ and the other one going from $(x_0,h_n(x_0))$ to  $(b,h(b))$. The  set  $\partial \Omega_h \cap ([a_k,b_k] \times \mathbb{R})$ contains a curve $\gamma_k : (0,L_k) \to [a_k,b_k]\times \mathbb{R}$ parametrized by arc-length connecting the two endpoints, that is $ \gamma_k(0) = (a_k,h(a_k)), \gamma_k(L_k) = (b_k,h(b_k)) $ and (again using \cite{falconer1986geometry}, Lemma 3.12 and Theorem 3.8) $|\dot{\gamma}_k|=1$ $\mathcal{L}^1$-a.e. in $(0,L_k)$. We prolongate the line connecting $(a_k,h(a_k))$ and $(x_0,h_n(x_0))$ until it intersects the curve $\gamma$ say at $\gamma(t_0)$. Setting $\xi_k = \gamma(t_0)- (a_k,h(a_k))$ we have
\begin{align}\label{Boundaryhn1}
\begin{split}
|\xi_k| \varphi\left(\frac{\xi_k^\perp}{|\xi_k|}\right)=\varphi(\xi_k^\perp) = \varphi\Big(\left(\int_0^{t_0}\dot{\gamma}_k(t)\mathrm{d}t\right)^\perp\Big) \leq \int_0^{t_0}\varphi(\dot{\gamma}_k^\perp(t))\mathrm{d}t.
\end{split}
\end{align}
We define the curve $\tilde{\gamma}_k :(0,\tilde{L}_k) \to [a_k,b_k]\times\mathbb{R}$ parametrized by arc-length that connects $(x_0,h(x_0))$ with $\gamma_k(t_0)$ through a straight line and then follows $\gamma_k$ until $(b_k,h(b_k))$. Setting $\zeta_k = (b_k,h(b_k)) -(x_0,h_n(x_0))$ we obtain
\begin{align}\label{Boundaryhn2}
\begin{split}
|\zeta_k| \varphi\left(\frac{\zeta_k^\perp}{|\zeta_k|}\right)&=\varphi(\zeta_k^\perp) = \varphi\Big(\left(\int_0^{\tilde{L}_k}\dot{\tilde{\gamma}}_k(t)\mathrm{d}t\right)^\perp\Big) \leq \int_0^{\tilde{L}_k}\varphi(\dot{\tilde{\gamma}}^\perp(t))\mathrm{d}t \\&= \int_0^{\tilde{L}_k-L_k+t_0}\varphi(\dot{\tilde{\gamma}}^\perp(t))\mathrm{d}t + \int_{\tilde{L}_k-L_k+t_0}^{\tilde{L}_k}\varphi(\dot{\tilde{\gamma}}^\perp(t))\mathrm{d}t\\&= |\xi_k-\zeta_k|\varphi\left(\frac{\xi_k^\perp}{|\xi_k|}\right) + \int_{t_0}^{L_k}\varphi(\dot{\gamma}^\perp(t))\mathrm{d}t.
\end{split}
\end{align}
Here we used the fact that $\xi_k,\zeta_k$ point in the same direction and $\tilde{\gamma}$ connects $(x_0,h(x_0))$ with $\gamma_k(t_0)$ through a straight line and that $\tilde{\gamma}_k(\tilde{L}_k-L_k+t)=\gamma_k(t), t \in (t_0,L_k)$. Using  (\ref{Boundaryhn1}) and (\ref{Boundaryhn2}) we obtain
\begin{align}\label{energyinequality2}
\begin{split}
\int_0^{L_k}\varphi(\dot{\gamma}_k^\perp(t))\mathrm{d}t&\geq |\zeta_k| \varphi\left(\frac{\zeta_k^\perp}{|\zeta_k|}\right) + (|\xi_k|-|\xi_k-\zeta_k|)\varphi\left(\frac{\xi_k^\perp}{|\xi_k|}\right) \\&=\int_{\partial \Omega_{h_n} \cap [a_k,b_k]\times\mathbb{R}}\varphi(\nu)\mathrm{d}\mathcal{H}^1.
\end{split}
\end{align}
The last inequality follows, since the first term on the left side is equal to the integration of $\varphi$ over the line segment connecting $(x_0,h(x_0))$ and $(b_k,h(b_k))$, while $(|\xi_k|-|\xi_k-\zeta_k|) = || (x_0,h(x_0))-(a_k,h(a_k))||_2$ and the second term equals therefore the integration of $\varphi$ over the line segment connecting $(a_k,h(a_k))$ and $(x_0,h(x_0))$. Summing over all $k \in \{1,\ldots,K\}$ we obtain
\begin{align*}
\int_{\partial \Omega_{h_n} \setminus \partial \Omega_h} \varphi(\nu)\mathrm{d}\mathcal{H}^1 \leq \sum_{k=1}^K \int_0^{L_k}\varphi(\dot{\gamma}_k^\perp(t))\mathrm{d}t.
\end{align*}
Note that 
\begin{align*}
\int_0^{L_k}\varphi(\dot{\gamma}_k^\perp(t))\mathrm{d}t \leq \int_{\partial \Omega_h \cap [a_k,b_k]\times \mathbb{R}} \varphi(\nu)\mathrm{d}\mathcal{H}^1.
\end{align*}
Summing over $k$ we obtain
\begin{align*}
\int_{\partial\Omega_{h_n}\setminus \Omega_h} \varphi(\nu) \mathrm{d}\mathcal{H}^1 \leq \sum_k \int_{\partial \Omega_h \cap [a_k,b_k]\times \mathbb{R}} \varphi(\nu)\mathrm{d}\mathcal{H}^1.
\end{align*}
Now for any $k$ such that $\gamma_k$ used above for estimating the energy in $[a_k,b_k]$ contains a vertical component connecting $(b_k,y_1)$ with $(b_k,h(b_k))$ and $\gamma_{k+1}$ contains a vertical component connecting $(b_k,h(b_k)) = (a_{k+1},h(a_{k+1}))$ with $(a_{k+1},y_2)$ we show that the line segment $ [(b_k,h(b_k)),(b_k,\min \{y_1,y_2\})] \subset \partial \Omega_{h} \setminus \partial\overline{\Omega}_h$. Since both the segments are contained in a continuous curve $\gamma_k,\gamma_{k+1}$ respectively we have that $h^-(b_k) \geq \min\{y_1,y_2\}>h(b_k)$. It therefore follows that $ [(b_k,h(b_k)),(b_k,\min \{y_1,y_2\})] \subset \partial \Omega_{h} \setminus \partial\overline{\Omega}_h$. Since $\varphi$ is even we have
\begin{align}\label{Energyinequality}
\begin{split}
\int_{\partial\Omega_{h_n}\setminus \Omega_h} \varphi(\nu) \mathrm{d}\mathcal{H}^1 &\leq \sum_k \int_{\partial \Omega_h \cap (a_k,b_k)\times \mathbb{R}} \varphi(\nu)\mathrm{d}\mathcal{H}^1+ 2\int_{\partial \Omega_h \setminus \partial \overline{\Omega}_h} \varphi(\nu)\mathrm{d}\mathcal{H}^1 \\&\leq\int_{(\partial \Omega_h \cap\partial\overline{\Omega}_h)\setminus \partial \Omega_{h_n}}\varphi(\nu)\mathrm{d}\mathcal{H}^1 +2\int_{\partial \Omega_h \setminus \partial \overline{\Omega}_h} \varphi(\nu)\mathrm{d}\mathcal{H}^1.
\end{split}
\end{align}
Using (\ref{EnergyEquality}),(\ref{Energyequality2}) and (\ref{Energyinequality}) we obtain
\begin{align*}
E^S(h_n) \leq E^S(h).
\end{align*}
Taking the $\limsup$ as $n \to +\infty$ yields the claim.
\end{proof}

 The following lemma is needed in order to match the volume constraint and was established in \cite[Lemma 3.1]{DP1}. 


\begin{lemma}\label{Vitali_lemma} 
Let  $h_n\in L^1(S^1_L;\mathbb{R}_+)$ such that $h_n\to h$ in $L^1(S^1_L)$. For every sequence $\{\lambda_n\}$ converging to 0, there exist a constant $\mu>0$ (depending on the sequences $\{\lambda_n\}$, $\{h_n\}$, and $h$) and an integer $N_{\mu}$ such that 
\begin{equation}\label{lowerbound}
|H_{\lambda_n}|\,+\,\frac{1}{\lambda_n}\int_{[0,L]\setminus H_{\lambda_n}}h_n(x_1)\,\mathrm{d}x_1>\mu
\end{equation}
for every $n\geq N_{\mu}$, where $H_{\lambda_n}:=\{x_1\in[0,L]\,:\, h_n(x_1)\geq\lambda_n\}$.
\end{lemma}

\begin{proof}
By contradiction, up to passing to a subsequence (not relabeled) both for $\{\lambda_n\}$ and $\{h_n\}$ we have that 
\begin{equation}\label{lowerbound_contradicted}
|H_{\lambda_n}|\,+\,\frac{1}{\lambda_n}\int_{[0,L]\setminus H_{\lambda_n}}h_n(x_1)\,\mathrm{d}x_1\leq \mu_n
\end{equation}
for some $\mu_n$ converging to zero.
Fixed $\eta\in(0,\|h\|_{L^1(S^1_L)})$. By Vitali's Theorem there exists $\mu_{\eta}>0$ such that 
$$
\|h_n\|_{L^1(S)}\leq \|h\|_{L^1(S^1_L)}-\eta
$$
for every measurable set $S$ with $|S|\leq\mu_{\eta}$ and $n\in \mathbb{N}$.  From \eqref{lowerbound_contradicted} it follows that 
$|H_{\lambda_n}|\leq\mu_{\eta}$ for $n$ large enough, and hence we obtain that
\begin{equation}\label{bound}
\|h_n\|_{L^1(H_{\lambda_n})}\leq \|h\|_{L^1(S^1_L)}-\eta.
\end{equation}
However, by \eqref{lowerbound_contradicted} we also have that 
\begin{align*}
0\leftarrow\mu_n\geq \frac{1}{\lambda_n}\int_{[0,L]\setminus H_{\lambda_n}}h_n(x_1)\,\mathrm{d}x_1
&=\frac{1}{\lambda_n}\left(\|h_n\|_{L^1(S^1_L)}- \|h_n\|_{L^1(H_{\lambda_n})}\right)\\
&\geq  \frac{1}{\lambda_n}\left(\|h_n\|_{L^1(S^1_L)}-\|h\|_{L^1(S^1_L)})+\eta\right)
\end{align*}
where we used \eqref{bound} in the last inequality. We reached a contradiction with the fact that
$$
 \frac{1}{\lambda_n}\left(\|h_n\|_{L^1(S^1_L)}-\|h\|_{L^1(S^1_L)}+\eta\right)\to+\infty
$$
since $\lambda_n\to0$ and $h_n\to h$ in $L^1(S^1_L)$.
\end{proof}

In order to prove the upper bound we need the following density result whose proof relies on  some ideas of both \cite{DP1} and \cite{fonseca2007equilibrium},  also adapted to the anisotropic case. 

\begin{proposition}\label{Prop SurfaceContinuousApprox_withVolume} 
Let $(y,u,h) \in X$ such that $||h||_{L^1(S^1_L)} = V$. Then there exists a sequence of Lipschitz functions $h_n :S^1_L \to \mathbb{R}_+$ such that $||h_n||_{L^1(S^1_L)} = V$, and a sequence $u_n\in H^1(\Omega_{h_n};\mathbb{R}^2)$ such that:
\begin{itemize}
\item[(i)] $(y,u_n,h_n)\to(y,u,h)$ in $X$ as $n\to+\infty$;
\item[(ii)] $E(y,u_n,h_n)\to E(y,u,h)$ as $n\to+\infty$.
\end{itemize}
\end{proposition}

\begin{proof} Without loss of generality we assume that 
\begin{equation}\label{finite_energy}
E(y,u,h) < +\infty.
\end{equation}
By Lemma \ref{Prop SurfaceContinuousApprox} there exists a sequence of Lipschitz functions $\tilde{h}_n :[0,L] \to \mathbb{R}_+$ such that $\tilde{h}_n\to h$ in $L^1(S^1_L)$, $\tilde{h}_n\leq \tilde{h}_{n+1}\leq h$, 
\begin{equation}\label{epigraph} 
Q \setminus \Omega_{\tilde{h}_n} \to Q \setminus \Omega_h
\end{equation}
 with respect to the Hausdorff-distance,  and 
\begin{equation}\label{surface_energy_convergence2} 
E^S(\tilde{h}_n) \to E^S(h).
\end{equation} 
Note that the functions $\tilde{h}_n$ do not satisfy the $L^1$-constraint.  We introduce a parameter $\lambda_n$ to measure how much each $\tilde{h}_n$ differs from such $L^1$-constraint, as
\begin{equation}\label{lambdan} 
\lambda_n:=\left(V- ||\tilde{h}_n||_{L^1(S^1_L)}\right)^{\beta}\geq0
\end{equation}
for every $n\in\mathbb{N}$, where $\beta$ is fixed number in $(0,1)$.  We then define the functions $h_n :[0,L] \to \mathbb{R}_+$ by
\begin{equation}\label{hndefinition} 
h_n(x_1):=
\begin{cases}
\tilde{h}_n(x_1) &\textrm{if $\tilde{h}_n(x_1)=0$},\\
\tilde{h}_n(x_1)+\varepsilon_n &\textrm{if $\tilde{h}_n(x_1)\geq\lambda_n$},\\
\left(1+\frac{\varepsilon_n}{\lambda_n}\right)\tilde{h}_n(x_1)&\textrm{if $\tilde{h}_n(x_1)\in(0,\lambda_n)$}
\end{cases}
\end{equation}
for every $x_1\in[0,L]$ and for a number $\varepsilon_n\geq0$ which will be chosen so that $||h_n||_{L^1(S^1_L)}=V$. More precisely, by a straightforward computation one finds
\begin{equation}\label{epsilonn} 
\varepsilon_n:=\frac{1}{\mu_n}\left(V- ||\tilde{h}_n||_{L^1(S^1_L)}\right),
\end{equation}
where $\mu_n$ is given by
$$
\mu_n:=|\tilde{H}_{\lambda_n}|\,+\,\frac{1}{\lambda_n}\int_{[0,L]\setminus \tilde{H}_{\lambda_n}}\tilde{h}_n(x_1)\,\mathrm{d}x_1
$$
with $\tilde{H}_{\lambda_n}:=\{x_1\in[0,L]\,:\, \tilde{h}_n(x_1)\geq\lambda_n\}$. Since $\lambda_n\to0$ by the $L^1$-convergence of the $\tilde{h}_n$, we can applied Lemma \ref{Vitali_lemma} and obtain a $\mu>0$ and an integer $N_{\mu}$ such that 
$$
\mu_n>\mu
$$
for every $n\geq N_{\mu}$. Then, from \eqref{epsilonn} it easily follows that
\begin{equation}\label{epsilonnzero}
0\leq\varepsilon_n\leq\frac{1}{\mu}\left(V- ||\tilde{h}_n||_{L^1(S^1_L)}\right)\to0,
\end{equation}
and 
\begin{equation}\label{ratiozero}
0\leq\frac{\varepsilon_n}{\lambda_n}\leq\frac{1}{\mu}\left(V- ||\tilde{h}_n||_{L^1(S^1_L)}\right)^{1-\beta}\to0
\end{equation}
since $\beta\in(0,1)$. Note that in particular \eqref{epigraph} together with \eqref{epsilonnzero}  and $\tilde{h}_n \leq h_n \leq \tilde{h}_n +\varepsilon_n$   implies that
\begin{equation}\label{first_assert_half}
Q \setminus \Omega_{h_n} \to Q \setminus \Omega_h
\end{equation}
 with respect to the Hausdorff-distance. Furthermore, by also employing Bolzano's Theorem one can prove that 
$$
 |h_n(x_1)-h_n(x'_1)|\leq C_n\left(1+\frac{\varepsilon_n}{\lambda_n}\right)|x_1-x'_1| 
$$
for every $x_1,x'_1\in[0,L]$, where $C_n>0$ denotes the Lipschitz constant associated to $\tilde{h}_n$ and hence, $h_n$ are also Lipschitz. We now prove that 
\begin{equation}\label{surface_energy_convergence} 
E^S(h_n) \to E^S(h).
\end{equation}
By \eqref{ESenergy} and \eqref{hndefinition} we have that 
\begin{align}
|E^S(\tilde{h}_n)-E^S(h_n)|&\leq \gamma_f\left|\int_{\partial \Omega_{\tilde{h}_n} \cap \Omega_{\tilde{h}_n}^{1/2}\cap \{\lambda_n+\varepsilon_n>x_2>0\}} \varphi(\nu)\mathrm{d}\mathcal{H}^1-\int_{\partial \Omega_{h_n} \cap \Omega_{h_n}^{1/2}\cap \{\lambda_n>x_2>0\}} \varphi(\nu)\mathrm{d}\mathcal{H}^1\right|\nonumber\\
&= \gamma_f\sum_{i\in\mathcal{I}^n}\int_{a^n_i}^{b^n_i} \left| \varphi(\nu(\tilde{h}_n))\sqrt{1+(\tilde{h}'_n)^2}-\varphi(\nu(h_n))\sqrt{1+(h'_n)^2}\right| \,\mathrm{d}x_1
\nonumber\\
& \leq \gamma_f\sum_{i\in\mathcal{I}^n}\left(\int_{a^n_i}^{b^n_i}   A^n   \,\mathrm{d}x_1 + \int_{a^n_i}^{b^n_i}   B^n   \,\mathrm{d}x_1\right) \label{surface_energy_difference}
\end{align}
for some index set $\mathcal{I}^n$ and $a_i^n<b_i^n$  with $(a_i^n,b_i^n) \cap (a_j^n,b_j^n)= \emptyset$ for all $i,j \in \mathcal{I}^n, i \neq j$  such that 
$$
\bigcup_{i\in\mathcal{I}^n}(a_i^n,b_i^n)= \{x_1\in[0,L]\,:\, 0<\tilde{h}_n(x_1)<\lambda_n\},
$$
$$
A^n:=\varphi(\nu(h_n))\left| \sqrt{1+(\tilde{h}'_n)^2}-\sqrt{1+(h'_n)^2}\right|,
$$
and
$$
B^n:= \sqrt{1+(\tilde{h}'_n)^2} \left| \varphi(\nu(\tilde{h}_n))-\varphi(\nu(h_n))\right|,
$$
 where $\nu(g)$ for a Lipschitz function $g:[0,L] \to \mathbb{R}_+$ here denotes the normal on the graph of $g$ with respect to the graph parametrization, i.e.,
\begin{equation}\label{normal_graph} 
\nu(g)(x_1):=\frac{(-g'(x_1),1)}{\sqrt{1+(g'(x_1))^2}}
\end{equation}
for every $x_1\in[0,L]$. We observe that 
\begin{align}
A^n&\leq C \frac{\sigma_n(\tilde{h}'_n)^2}{\sqrt{1+(\tilde{h}'_n)^2}+\sqrt{1+(h'_n)^2}}\nonumber\\
&\leq\frac{C\sigma_n(\tilde{h}'_n)^2}{2\sqrt{1+(\tilde{h}'_n)^2}}\leq\frac{C}{2}\sigma_n\sqrt{1+(\tilde{h}'_n)^2},
 \label{An}
\end{align}
where in the first inequality we used \eqref{phinorm} and
$$\sigma_n:=\left[\left(\frac{\varepsilon_n}{\lambda_n}\right)^2+2\frac{\varepsilon_n}{\lambda_n}\right].$$
 We now observe that for every $x_1\in\bigcup_{i\in\mathcal{I}^n}(a_i^n,b_i^n)$
\begin{align}
|\nu(\tilde{h}_n)-\nu(h_n)|&\leq\frac{|(-\tilde{h}'_n\sqrt{1+(h'_n)^2}+h'_n\sqrt{1+(\tilde{h}'_n)^2}, \sqrt{1+(h'_n)^2}-\sqrt{1+(\tilde{h}'_n)^2})|}{1+(h_n'(x_1))^2}\nonumber\\
&\leq\frac{1}{\sqrt{1+(h_n'(x_1))^2}}\sqrt{\left(\frac{\sigma_n}{2}+\frac{\varepsilon_n}{\lambda_n}|h'_n|\right)^2+\frac{\sigma_n^2}{4}}
\nonumber\\
&\leq\frac{1}{\sqrt{2}}\left[\left(\frac{\varepsilon_n}{\lambda_n}\right)^2+3\frac{\varepsilon_n}{\lambda_n}\right]\leq \sqrt{2}\,\sigma_n,\label{normals}
\end{align}
where in the first inequality we used \eqref{normal_graph}, and in second the fact that $h'_n(x_1)=\tilde{h}'_n(x_1)+\frac{\varepsilon_n}{\lambda_n}\tilde{h}'_n(x_1)$ and the same reasoning employed in \eqref{An}.
Therefore, by \eqref{ratiozero} and \eqref{normals} we have that $|\nu(\tilde{h}_n)-\nu(h_n)|\to0$ as $n\to+\infty$. Since $\varphi$ is bounded and convex and hence, Lipschitz, we then obtain that
\begin{align}
B^n&\leq C_{\varphi} \sqrt{1+(\tilde{h}'_n)^2}|\nu(\tilde{h}_n)-\nu(h_n)|\nonumber\\
&\leq \sqrt{2}\,C_{\varphi}\sigma_n\sqrt{1+(\tilde{h}'_n)^2}, \label{Bn}
\end{align}
for large $n\in\mathbb{N}$, where $C_{\varphi}$ is the Lipschitz constant of $\varphi$. Therefore, from \eqref{surface_energy_difference}, \eqref{An}, and \eqref{Bn} it follows that there exists a constant $C'>0$ such that
$$
|E^S(\tilde{h}_n)-E^S(h_n)|\leq C'\gamma_f \sigma_n \sum_{i\in\mathcal{I}^n}\int_{a^n_i}^{b^n_i}   \sqrt{1+(\tilde{h}'_n)^2}   \,\mathrm{d}x_1\leq C'\gamma_f \sigma_n \mathcal{H}^1(\Gamma_n),
$$
where $\Gamma_n$ denotes here the graph of $\tilde{h}_n$.  Therefore, since  $\mathcal{H}^1(\Gamma_n)$ are uniformly bounded because of \eqref{phinorm} and \eqref{surface_energy_convergence2}, and $\sigma_n$ tends to 0 by \eqref{ratiozero}, we can conclude that
$$
|E^S(\tilde{h}_n)-E^S(h_n)|\to 0
$$
from which \eqref{surface_energy_convergence} follows in view of  \eqref{surface_energy_convergence2}. 

Let us now define $u_n: \Omega_{h_n}\to\mathbb{R}^2$ by
\begin{equation}\label{undefinition} 
u_n(x_1,x_2):=
\begin{cases}
u(x_1,x_2-\varepsilon_n) &\textrm{if $x_2>x^0_2+\varepsilon_n$},\\
u(x_1,x^0_2) &\textrm{if $x^0_2+\varepsilon_n\geq x_2>x^0_2$},\\
u(x_1,x_2) &\textrm{if $x^0_2\geq x_2$,}
\end{cases}
\end{equation}
for a $x^0_2\in(-R,-\varepsilon_n)$ chosen in such a way that $u(\cdot,x^0_2)\in H^1(S^1_L;\mathbb{R}^2)$ (which we can by \eqref{finite_energy}). Note that $u_n$ are well defined in $\Omega_{h_n}$ since $h_n\leq \tilde{h}_n+\varepsilon_n\leq h+\varepsilon_n$. Furthermore, by  \eqref{epsilonnzero} and \eqref{undefinition} we have that $u_n\to u$ in $L^2_{\rm loc}(\Omega_{h_n};\mathbb{R}^2)$ as $n\to+\infty$, which together with \eqref{first_assert_half} yields Assertion (i). 

The remaining part of the proof is devoted to prove 
\begin{equation}\label{elastic_energy_convergence} 
E^{el}(y,u_n,h_n) \to E^{el}(y,u,h)
\end{equation}
which together with \eqref{surface_energy_convergence} implies Assertion (ii). We begin by observing that 
 \begin{align}\label{elastic_energy_convergence_novolume}
 \begin{split}
\lim_{n\to +\infty} E^{el}(y,u,\tilde{h}_n) &= \lim_{n\to +\infty}\int_{\Omega_{\tilde{h}_n}} W_y(x,Eu(x))\mathrm{d}x \\&=  \int_{\Omega_{h}} W_y(x,Eu(x))\mathrm{d}x= E^{el}(y,u,h). 
\end{split}
\end{align}
where we applied the Monotone Convergence Theorem. Furthermore, by \eqref{undefinition} we observe that
 \begin{align}
 |E^{el}(y,u_n,\tilde{h}_n)-& E^{el}(y,u,\tilde{h}_n)|\leq \,C\Big[\,\int_{[0,L]\times[x^0_2,x^0_2+\varepsilon_n]} W_y(x,Eu(x_1,x_2^0))\mathrm{d}x\nonumber\\
  & + \int_{\{\tilde{h}_n>0\}\times[0,\varepsilon_n]}  |\delta\nabla y|^2   \,\mathrm{d}x
+ \int_{\{\tilde{h}_n=0\}\times[-\varepsilon_n,0]} W_y(x,Eu(x))\mathrm{d}x
 \nonumber\\
  & \quad\qquad\qquad\qquad\qquad\qquad\qquad\qquad+ \int_{E_n} W_y(x,Eu(x))\mathrm{d}x  \,\Big]\nonumber\\
  & \leq C\Big[\,\varepsilon_n\int_{[0,L]} |Eu(x_1,x_2^0)|^2\mathrm{d}x_1
  +  \int_{[0,L]\times[0,\varepsilon_n]}  |\delta\nabla y|^2   \,\mathrm{d}x\nonumber\\
  & \quad\quad\quad+  \int_{[0,L]\times[-\varepsilon_n,0]}  |Eu|^2   \,\mathrm{d}x
  +  \int_{E_n} |Eu|^2+|\delta\nabla y|^2 \mathrm{d}x  \,\Big] \label{elastic_energy_convergence2} 
 \end{align}
 where $\{\tilde{h}_n=0\}:=\{x_1\in[0,L]\,:\, \tilde{h}_n(x_1)=0\}$, $\{\tilde{h}_n>0\}:=[0,L]\setminus\{\tilde{h}_n=0\}$, $\{\lambda_n>\tilde{h}_n>0\}:=\{x_1\in[0,L]\,:\, \lambda_n>\tilde{h}_n(x_1)>0\}$, and 
 $$
 E_n:=\left(\Omega_{\tilde{h}_n}\setminus(\Omega_{\tilde{h}_n}-\varepsilon_n\textbf{e}_2)\right)\cap\left(\{\lambda_n>\tilde{h}_n>0\}\times(0,+\infty)\right).$$
Notice that $|E_n|\leq C\varepsilon_n$ for some constant $C>0$, since by \eqref{hndefinition} $0<\tilde{h}_n-(h_n-\varepsilon_n)\leq +\varepsilon_n$ for $x_1\in\{\lambda_n>\tilde{h}_n>0\}$. Therefore,  from  \eqref{epsilonnzero} and \eqref{elastic_energy_convergence2} it easily follows  that 
$$
 |E^{el}(y,u_n,\tilde{h}_n)- E^{el}(y,u,\tilde{h}_n)|\to0
$$
as $n\to+\infty$, and hence, also in view of \eqref{elastic_energy_convergence_novolume}, we obtain \eqref{elastic_energy_convergence}. This  concludes the proof of Assertion (ii) and of the Proposition.
  \end{proof} 


 In view of previous density results, and in particular because we can reduce to Lipschitz profile functions (matching the volume constraint), by Korn's inequality  displacements of energy-bounded sequences are in  $W^{1,2}(\Omega_h)$ and hence, by Serrin's Theorem we can restrict to the case of smooth displacements $u$. In this regard, notice also that  due to the growth and continuity properties of $W_y$ strong convergence in $W^{1,2}(\Omega_h)$ implies convergence of the elastic energies.  For such pairs of regular profiles and displacement we construct the recovery sequence explicitly. The convergence of the elastic energy then follows by Taylor expanding the discrete cell-energy and the convergence of the surface energy can also be reduced to the computation of the surface energy of (a part of) a half space. 


\begin{proposition}\label{Proposition Gamma-limsup} Let $(y,u,h) \in X$ we then have 
\begin{align*}
E''(y,u,h) \leq E(y,u,h).
\end{align*}
\end{proposition}
\begin{proof} We use a density argument. Note that for general $(y,u,h) \in X$ it is not true, that for any such $u$ there exist $\{u_k\} \subset C^\infty(\Omega_h)$ and such that $u_k \to u $ strongly in $H^1(\Omega_h)$, since $h$ might not be Lipschitz continuous. However, by Proposition \ref{Prop SurfaceContinuousApprox_withVolume} there exists a sequence of Lipschitz functions $h_n :S^1_L \to \mathbb{R}_+$ such that $||h_n||_{L^1(S^1_L)} = V$, and a sequence $u_n\in H^1(\Omega_{h_n};\mathbb{R}^2)$ such that $(y,u_n,h_n)\to(y,u,h)$ in $X$ and  $E(y,u_n,h_n)\to E(y,u,h)$ as $n\to+\infty$.
 
We can therefore assume that $h \in \mathrm{Lip}([0,L])$. 

Now by \cite{fonseca2007equilibrium}, Theorem 4.2 there exists a $A= A(u) \in \mathbb{R}^{2\times 2}$ such that $A=-A^T$ and a constant $C=C(h)$ such that
\begin{align*}
\int_{\Omega_h}|\nabla u-A|^2\mathrm{d}x \leq C\int_{\Omega_h}|Eu|^2\mathrm{d}x,
\end{align*}
moreover by Poincare\'es Inequality there exist a vector $b=b(u) \in \mathbb{R}^2$ such that
\begin{align*}
\int_{\Omega_h}|u-(Ax+b)|^2\mathrm{d}x \leq C\int_{\Omega_h}|\nabla u-A|^2\mathrm{d}x,
\end{align*}
with $C=C(h)$ as above. We therefore obtain that $v(x) = u(x)-(Ax+b) \in H^1(\Omega_h)$ and with that $u \in H^1(\Omega_h)$.
Hence there exist $\{u_k\} \subset C^\infty(\Omega_h)$ converging to $u$ strongly in $H^1(\Omega_h)$. Since $W_y(x,\cdot)$ is continuous and satisfies
\begin{align*}
W_y(x,\xi) \leq C(1+|\xi|^2)
\end{align*}
we have that
\begin{align*}
\lim_{k\to+\infty} E^{el}(y,u_k,h) = E^{el}(y,u,h). 
\end{align*}
Further we  reduce  to the setting where $h:[0,L] \to \mathbb{R}$ is piecewise affine. Fix $\delta>0$. There exists $A_\delta \subset S^1_L$ open such that $\{h=0\} \subset A_\delta$ and 
\begin{align}\label{Adelta small}
|A_\delta|\leq |\{h=0\}| +\delta.
\end{align}
As $\{h=0\}$ is compact and $A_\delta$ is open we can assume that
\begin{align*}
A_\delta = \bigcup_{k=1}^{N_\delta}I_k^\delta,
\end{align*}
with $N_\delta \in \mathbb{N}$, $I_k^\delta$ open intervals and $\mathrm{dist}(I_k^\delta,I_j^\delta) \geq d_\delta >0$ for all $j\neq k$. Since the $I_k^\delta$ are disjoint and $\{h=0\} \subset A_\delta$, by (\ref{Adelta small}) we have that
\begin{align*}
\sum_{k=1}^{N_\delta}|I_k^\delta \setminus \{h=0\}| =\sum_{k=1}^{N_\delta}|I_k^\delta| -  |\{h=0\}|=|A_\delta|-|\{h=0\}| \leq \delta.
\end{align*}
Define $\delta_k = |I_k^\delta\setminus \{h=0\}|$ and $\eta = \min \{\delta_k,\frac{d_\delta}{2} : k=1,\ldots,N_\delta \}$. Now set $\Delta_\delta= \partial A_\delta \cup \partial (A_\delta)_{\eta}\cup \left( ([0,L] \setminus (A_\delta)_\eta )\cap \delta\mathbb{Z} \right) $ and we order $\Delta_\delta$ by writing $\Delta_\delta = \left\{x_i\right\}_{i=1}^{M_\delta}$ with $
0\leq x_1\leq \ldots \leq x_i \leq x_{i+1} \ldots\leq L
$. Now define $h_\delta :[0,L] \to \mathbb{R}_+$ by
\begin{align*}
h_\delta(x_i) = \begin{cases}0 &x_i \in \partial A_\delta, \\
h(x_i) &\text{otherwise}
\end{cases}
\end{align*}
and interpolate linearly between the function values of $h_\delta(x_i)$. Note that if  $h_\delta(x_i)\neq h(x_i)$ we have that $x_i \in \partial A_\delta$ and therefore $x_i \in \partial I_k^\delta$ for some $k=1,\ldots,N_\delta$. Since $|I_k^\delta\setminus \{h=0\}|=\delta_k$ there exists $x $ with $h(x)=0$ such that $|x-x_i|\leq \delta_k$. Since $h$ is Lipschitz have that $|h_\delta(x_i)-h(x_i)|\leq C\delta_k$. Now let $x \in S^1_L$ if $x \in I_k^\delta$ we have that $I_k^\delta \setminus \{h=0\}=\delta_k$, so there exist $x' \in \{h=0\}$, $|x-x'|\leq \delta_k$, we therefore have $h(x) \leq C\delta_k\leq \delta$. On the other hand if $ x \notin A_\delta $ there exists $x_i,x_{i+1} \in \Delta_\delta$ with $|x-x_i|,|x-x_{i+1}|\leq \delta$ and $|h(x_i)-h_\delta(x_i)| \leq C\delta$. By the Lipschitz continuity of both $h_\delta$ and $h$ we have that $|h(x)-h(x_\delta)|\leq C\delta$. We therefore have that $h_\delta \to h$ uniformly on $S^1_L$ and also $h_\delta \to h$ with respect to the convergence given in Definition \ref{Definition Convergence}. Note that $h_\delta$ is a piecewise affine function so we are done with the reduction to piecewise affine profiles if we show that
\begin{align*}
\lim_{\delta\to 0} E(y,u,h_\delta) = E(y,u,h).
\end{align*}
Since $Q\setminus \Omega_{h_\delta} \to Q\setminus \Omega_h$ we have that
\begin{align*}
\lim_{\delta \to 0} E^{el}(y,u,h_\delta) =  E^{el}(y,u,h).
\end{align*}
Moreover by the semicontinuity of the surface energy we have that
\begin{align*}
\liminf_{\delta \to 0} E^S(h_\delta)\geq E^S(h).
\end{align*}
It thus remains to show that
\begin{align*}
\limsup_{\delta \to 0} E^S(h_\delta)\leq E^S(h).
\end{align*}
Now we have by the Area formula and noting that $h^\prime_\delta=0$ we have that
\begin{align}\label{Ikdelta}
\begin{split}
\int_{\partial \Omega_{h_\delta} \cap(I_k^\delta\times \mathbb{R})}\varphi(\nu)\mathrm{d}\mathcal{H}^1 &= \int_{I_k^\delta}\varphi((1,h^\prime_\delta)^\perp)\mathrm{d}x \\&= \int_{I_k^\delta \cap \{h=0\}}\varphi((1,h^\prime_\delta)^\perp)\mathrm{d}x + \int_{I_k^\delta\setminus \{h=0\}}\varphi((1,h^\prime_\delta)^\perp)\mathrm{d}x \\&\leq \int_{I_k^\delta }\varphi((1,h^\prime)^\perp)\mathrm{d}x + C|I_k^\delta \setminus\{h=0\}|\\&=\int_{\partial \Omega_{h} \cap(I_k^\delta\times \mathbb{R})}\varphi(\nu)\mathrm{d}\mathcal{H}^1+C\delta_k.
\end{split}
\end{align}
Connecting $(x_i,h_\delta(x_i))$ with $(x_i,h(x_i))$, and  $(x_{i+1},h_\delta(x_{i+1}))$ with $(x_{i+1},h(x_{i+1}))$ noting $|h(x_i)-h_\delta(x_i)|,|h(x_{i+1})-h_\delta(x_{i+1})|\leq C\delta_k$ and using convexity of $\varphi$ we obtain
\begin{align}\label{Ikdeltaeta}
\begin{split}
\int_{\partial \Omega_{h_\delta} \cap((I_k^\delta)_\eta\setminus I_k^\delta\times \mathbb{R})}\varphi(\nu)\mathrm{d}\mathcal{H}^1 &\leq \int_{\partial \Omega_{h} \cap((I_k^\delta)_\eta\setminus I_k^\delta\times \mathbb{R})}\varphi(\nu)\mathrm{d}\mathcal{H}^1 + \int_{[(x_i,h_\delta(x_i)),(x_i,h(x_i))]}\varphi(\nu)\mathrm{d}\mathcal{H}^1\\&\quad+\int_{[(x_{i+1},h_\delta(x_{i+1})),(x_{i+1},h(x_{i+1}))]}\varphi(\nu)\mathrm{d}\mathcal{H}^1\\&\leq \int_{\partial \Omega_{h} \cap((I_k^\delta)_\eta\setminus I_k^\delta\times \mathbb{R})}\varphi(\nu)\mathrm{d}\mathcal{H}^1 +C\delta_k.
\end{split}
\end{align}
Now outside of $(A_\delta)_\eta$ by convexity we have for a subinterval $I$ of $S^1_L\setminus (A_\delta)_\eta$
\begin{align}\label{0LsetminusAdelta}
\begin{split}
\int_{\partial \Omega_{h_\delta} \cap(I\times \mathbb{R})}\varphi(\nu)\mathrm{d}\mathcal{H}^1&=\int_{I}\varphi((1,h^\prime_\delta)^\perp)\mathrm{d}x= \varphi\left(\int_{I}(1,h^\prime_\delta)^\perp\mathrm{d}x\right)\\&= \varphi\left(\int_{I}(1,h^\prime)^\perp\mathrm{d}x\right)\leq \int_{I}\varphi((1,h^\prime)^\perp)\mathrm{d}x.
\end{split}
\end{align}
Noting that $A_\delta$ is a finite union of intervals, so $S^1_L\setminus (A_\delta)_\eta$ is (up to finitely many points) the finite union of intervals with nonempty interior. Summing over $k$ and all intervals $I$ using (\ref{Ikdelta})-(\ref{0LsetminusAdelta}) and noting that $\sum_k \delta_k=\delta$ we obtain
\begin{align*}
E^S(h_\delta) \leq E^S(h) + C\sum_k \delta_k\leq E^S(h) +C\delta.
\end{align*}
 Repeating the same arguments as in Lemma 4.8 we can assume that $||h_\delta||_{L^1(S^1_L)}=V$. 
Taking the $\limsup$ as $\delta \to 0$ the claim follows.  We therefore have that for every $(y,u,h) \in X$ there exist $(y,u_k,h_k)\in X$ with $y=Rx+b$, $R\in SO(2),b \in \mathbb{R}^2$, $u_k \in C^\infty(\mathbb{R}^2)$ and $h_k$ piecewise affine,  $||h_k||_{L^1(S^1_L)}=V$  such that
\begin{align}\label{Density}
\lim_{k \to +\infty} E(y,u_k,h_k) = E(y,u,h).
\end{align}
It therefore suffices to construct the recovery sequence for $h:[0,L] \to \mathbb{R}_+$ piecewise affine, $u \in C^\infty(\Omega_h)$ and for $y = Rx + b$, with $R \in SO(2)$ and $b \in \mathbb{R}^2$. We extend $u$ to a function $u \in C^\infty(Q)$. In the choice of $h_\varepsilon$ there are two cases two consider, either $\gamma_f \leq \gamma_s$ or $\gamma_f >\gamma_s$. We define $h_\varepsilon : \frac{\sqrt{3}}{2}\varepsilon(\mathbb{Z}+\frac{1}{2})\cap S^1_L \to \varepsilon\mathbb{N}$ by 
\begin{align*}
h_\varepsilon(i) = 
\begin{cases}\varepsilon\big\lfloor\varepsilon^{-1} h\left(\frac{\sqrt{3}}{2}\varepsilon( i+\frac{1}{2})\right)\big\rfloor &i \text{ even}, \gamma_s < \gamma_f, \\
\varepsilon\big\lfloor\varepsilon^{-1} h\left(\frac{\sqrt{3}}{2}\varepsilon( i+\frac{1}{2})\right)\big\rfloor  + \frac{\varepsilon}{2} &i \text{ odd}, \gamma_s < \gamma_f, \\
\varepsilon\big\lfloor\varepsilon^{-1} h\left(\frac{\sqrt{3}}{2}\varepsilon ( i+\frac{1}{2})\right)\big\rfloor +\varepsilon &i \text{ even}, \gamma_s \geq \gamma_f, \\
\varepsilon\big\lfloor\varepsilon^{-1} h\left(\frac{\sqrt{3}}{2}\varepsilon ( i+\frac{1}{2})\right)\big\rfloor  + \frac{3\varepsilon}{2} &i \text{ odd}, \gamma_s \geq \gamma_f. \\
\end{cases}
\end{align*} 
and finally we define
\begin{align*}
y_\varepsilon(i) = (R i+b)  + \sqrt{\varepsilon}u(i),\quad i \in \mathcal{L}_\varepsilon(\Omega_{h_\varepsilon}).
\end{align*}
We then have that $u_\varepsilon(i) = u(i)$ for all $\mathcal{L}_\varepsilon(\Omega_{h_\varepsilon})$. Moreover we have that $y_\varepsilon \to y$ in $L^2_{\mathrm{loc}}(\Omega_h)$, $u_\varepsilon \to u$ in $L^2_{\mathrm{loc}}(\Omega_h)$ and $h_\varepsilon \to h$ in the sense of Definition \ref{Definition Convergence}. First note that $||\nabla u||_\infty \leq C$ and therefore we for $\varepsilon$ small enough
\begin{align}\label{Wcellbound}
W_{\varepsilon,cell}(\nabla y_\varepsilon|T,i) \leq C \varepsilon ||\nabla u||^2_\infty\leq C\varepsilon.
\end{align}
Furthermore since $h$ is Lipschitz we have for $\eta >0$ and $\varepsilon$ small enough that
\begin{align*}
\#\left(\mathcal{L}_\varepsilon(\Omega_{h_\varepsilon}) \cap \{x \in \mathbb{R}^2 : \mathrm{dist}(x,\partial\Omega_h)\leq \eta\}\right)\leq C\eta \varepsilon^{-2}.
\end{align*} 
Finally note that $\#((\partial^+\mathcal{L}_\varepsilon(\Omega^-)\cup \partial^-\mathcal{L}_\varepsilon(\Omega^-))_\eta )\leq C\eta\varepsilon^{-2}$. Splitting the interactions into $\mathcal{L}_\varepsilon(\Omega_{h_\varepsilon}) \cap \{x \in \mathbb{R}^2 : \mathrm{dist}(x,\partial\Omega_h)\leq \eta\}$, $(\partial^+\mathcal{L}_\varepsilon(\Omega^-)\cup \partial^-\mathcal{L}_\varepsilon(\Omega^-))_\eta$, $\mathcal{L}^\circ(\Omega_h^+)\cup \mathcal{L}^\circ(\Omega_h^-)$ using (\ref{Wcellbound}) and using the fact that $F \mapsto \varepsilon^{-1} W_{\varepsilon,cell}(R+\varepsilon F,i)$ converges uniformly on compact subsets to $W_y(i,Eu)$ for $i \in \mathcal{L}^\circ(\Omega_h^+)\cup \mathcal{L}^\circ(\Omega_h^-)$ we obtain that
\begin{align*}
\limsup_{\varepsilon\to 0} E^{el}_\varepsilon(y_\varepsilon,u_\varepsilon,h_\varepsilon)\leq E^{el}(y,u,h) + C\eta.
\end{align*}
The claim for the elastic energy follows by taking $\eta \to 0$. Now since $h$ is a piecewise affine function we have that $\Omega_h$ has a polygonal boundary. We proof the inequality for an interval $I$ in which $h'=const$. First note that for such an interval it is easy to check that
\begin{align*}
\sum_{j \in \mathcal{L}_\varepsilon(\Omega_{h_\varepsilon}\cap(I\times\mathbb{R}))}\varepsilon(6-\#\mathcal{N}_\varepsilon(j)) \leq \int_{\partial \Omega_h \cap (I\times\mathbb{R})}\varphi(\nu)\mathrm{d}\mathcal{H}^1.
\end{align*}
We distinguish between the case $\gamma_f \leq \gamma_s$ and the case $\gamma_f > \gamma_s$.
 In the case $\gamma_f \leq \gamma_s$ we have that $h_\varepsilon(i) \geq \varepsilon$ for all $i \in \frac{\sqrt{3}}{2}\varepsilon(\mathbb{Z}+\frac{1}{2})\cap S^1_L$ and therefore
 \begin{align*}
 E^S_\varepsilon(h_\varepsilon) \leq E^S(h) + C\varepsilon,
 \end{align*}
 where the error $C\varepsilon$ is due to the finite number of points, where $h'$ jumps. In the case $\gamma_f > \gamma_s$ the same argument shows that 
 \begin{align}\label{Case2hneq0}
\gamma_f\sum_{j \in \mathcal{L}_\varepsilon(\Omega_{h_\varepsilon}\cap(I\times\mathbb{R}))}\varepsilon(6-\#\mathcal{N}_\varepsilon(j)) \leq \gamma_f\int_{\partial \Omega_h \cap (I\times\mathbb{R})}\varphi(\nu)\mathrm{d}\mathcal{H}^1,
\end{align}
for all $h$ such that $h>0$. On the other hand if there is an interval for which $h=0$ by the definition of $h_\varepsilon$ we have that also $h_\varepsilon=0$ and we have
\begin{align}\label{Caseh=0}
\gamma_s\sum_{j \in \mathcal{L}_\varepsilon(\Omega_{h_\varepsilon}\cap(I\times\mathbb{R}))}\varepsilon(6-\#\mathcal{N}_\varepsilon(j)) \leq \gamma_s\int_{\partial \Omega_h \cap (I\times\mathbb{R})}\varphi(\nu)\mathrm{d}\mathcal{H}^1.
\end{align}
Using (\ref{Case2hneq0}) and (\ref{Caseh=0}) we obtain as before
\begin{align*}
E^S_\varepsilon(h_\varepsilon) \leq E^S(h) +C\varepsilon.
\end{align*}
 It remains to modify $h_\varepsilon$ such that $||h_\varepsilon||_{L^1(S^1_l)}=V_\varepsilon$. By Lemma 2.5 in \cite{fonseca2007equilibrium} we have that $h_\varepsilon \to h$ in $L^1(S^1_L)$. We therefore have that $||h_\varepsilon||,V_\varepsilon \in \frac{\sqrt{3}}{2}\varepsilon^2\mathbb{N}$ and
\begin{align}\label{eq : Nvareps}
\frac{\sqrt{3}}{2} N_\varepsilon \varepsilon^2=V_\varepsilon - ||h_\varepsilon||_{L^1(S^1_L)} = V -V_\varepsilon + V- ||h_\varepsilon||_{L^1(S^1_L)}   \to 0.
\end{align}
This implies $|N_\varepsilon|^{\frac{1}{2}} \varepsilon \to 0$.
Now fix an interval $(a,b) \subset S^1_L$ such that $h'=\mathrm{const}$, $h>0$ on $(a,b)$ and fix  $i_\varepsilon \in \varepsilon\frac{\sqrt{3}}{2}(\mathbb{Z} +\frac{1}{2}) \cap (a,b)$ such that 
\begin{align*}
\mathrm{dist}(i_\varepsilon,\{a,b\}) \geq \frac{\sqrt{3}}{2} |N_\varepsilon|^{\frac{1}{2}} \varepsilon.
\end{align*}
 Set $I_\varepsilon = \left(i_\varepsilon -  \frac{\sqrt{3}}{2} \lfloor |N_\varepsilon|^{\frac{1}{2}}\rfloor \varepsilon,i_\varepsilon + \frac{\sqrt{3}}{2} (\lfloor |N_\varepsilon|^{\frac{1}{2}}\rfloor-1) \varepsilon\right)$ and define $\tilde{h}_\varepsilon : \frac{\sqrt{3}}{2}\varepsilon(\mathbb{Z}+\frac{1}{2}) \cap S^1_L \to \varepsilon\mathbb{N}$
\begin{align*}
\tilde{h}_\varepsilon(i) = \begin{cases} h_\varepsilon(i) & i \notin I_\varepsilon,\\
h_\varepsilon(i_\varepsilon) +\varepsilon\left(\mathrm{sign}(N_\varepsilon) \lfloor |N_\varepsilon|^{\frac{1}{2}}\rfloor + N_\varepsilon - \mathrm{sign}(N_\varepsilon) \lfloor |N_\varepsilon|^{\frac{1}{2}}\rfloor^2 \right)   & i=i_\varepsilon,\\
h_\varepsilon(i_\varepsilon) +\varepsilon \,\mathrm{sign}(N_\varepsilon) \lfloor |N_\varepsilon|^{\frac{1}{2}}\rfloor   & \text{otherwise.}
\end{cases}
\end{align*}
Now it is clear that $||\tilde{h}_\varepsilon||_{L^1(S^1_L)} = V_\varepsilon$. Furthermore 
\begin{align*}
E_\varepsilon^S(\tilde{h}_\varepsilon)-E_\varepsilon^S(h_\varepsilon) \leq  C\varepsilon \sum_{j \in \mathcal{L}_\varepsilon(\Omega_{\tilde{h}_\varepsilon} \cap (I_\varepsilon \times \mathbb{R}))} (6- \#\mathcal{N}_\varepsilon(j)) \leq C(h) \varepsilon |N_\varepsilon|^{\frac{1}{2}} \to 0.
\end{align*}
The last estimate follows by elementary geometry and estimating the number of points in $\mathcal{L}_\varepsilon$ that have edges crossing the boundary of the parallelogram with slope determined by $h'$.

This yields the claim for $(y,u,h) \in X$, $y=Rx +b$, with $R \in SO(2),b \in \mathbb{R}^2$, $h$ piecewise affine and $u \in C^\infty(\mathbb{R}^2)$.
Now by (\ref{Density}) for general $(y,u,h) \in X$ there exist $(y,u_k,h_k) \to (y,u,h)$ such that
\begin{align*}
\lim_{k \to +\infty} E(y,u_k,h_k) = E(y,u,h).
\end{align*}
Since
\begin{align*}
E''(y,u_k,h_k)\leq E(y,u_k,h_k).
\end{align*}
we have 
\begin{align*}
E''(y,u,h)\leq \liminf_{k\to +\infty} E''(y,u_k,h_k) \leq \liminf_{k\to +\infty} E(y,u_k,h_k) = E(y,u,h)
\end{align*}
and the claim follows.
\end{proof}

\section*{Acknowledgements}

P. Piovano acknowledges support from the Vienna Science and Technology Fund (WWTF), the City of Vienna, and Berndorf Privatstiftung under Project MA16-005, and the Austrian Science Fund (FWF) project P~29681.

\end{document}